\documentclass[11pt,reqno]{amsart}
\usepackage{diagrams}
\usepackage{amssymb}
\usepackage[pagewise]{lineno}
\numberwithin{equation}{section}
\usepackage{cite}
\usepackage{url}

\theoremstyle{definition}
\newtheorem{thm}{Theorem}[section]
\newtheorem{cor}[thm]{Corollary}
\newtheorem{lem}[thm]{Lemma}

\newtheorem{prop}[thm]{Proposition}
\newtheorem{defi}[thm]{Definition}
\newtheorem{rem}[thm]{Remark}
\newtheorem{note}[thm]{Notation}

\DeclareMathOperator{\Hc}{\mathcal{H}om}

\DeclareMathOperator{\Ima}{\mathrm{Im}}

\DeclareMathOperator{\p3}{\mathbb{P}^3}

\DeclareMathOperator{\pr}{\mathrm{pr}}

\DeclareMathOperator{\mo}{\mathcal{O}}

\newcommand{\mr}[1]{\mathrm{#1}}
\newcommand{\mb}[1]{\mathbb{#1}}
\newcommand{\mbf}[1]{\mathbf{#1}}
\newcommand{\mc}[1]{\mathcal{#1}}
\newcommand{\ov}[1]{\overline{#1}}

\newcommand{\mf}[1]{\mathfrak{#1}}

\begin{document}

\title[Hodge conjecture]
{Hodge conjecture for the moduli space of semi-stable sheaves over a nodal curve}

\author[A. Dan]{Ananyo Dan}

\address{School of Mathematics and Statistics, University of Sheffield, Hicks building, Hounsfield Road, S3 7RH, UK}

\email{a.dan@sheffield.ac.uk}

\author[I. Kaur]{Inder Kaur}

\address{Institut für Mathematik,
Goethe-Universität Frankfurt,
Robert-Mayer-Str. 6-8,
60325 Frankfurt am Main, Germany}

\email{kaur@math.uni-frankfurt.de}

\subjclass[2010]{$14$C$30$, $32$S$35$, $14$D$07$, $32$G$20$,  $14$D$20$, $14$D$22$, $14$D$05$}


\date{\today}

\begin{abstract}
 In this article, we prove the Hodge conjecture for a desingularization of the moduli space of rank $2$, semi-stable, torsion-free 
 sheaves with fixed odd degree determinant over a very general irreducible nodal curve of genus at least $2$.
 We also compute the algebraic Poincar\'{e} polynomial of the associated cohomology ring.
 \end{abstract}

\maketitle

\section{Introduction}
The Hodge conjecture is one of the outstanding problems of present day mathematics. Although it has been known for over seventy years, the evidence for it has been rather limited (see \cite{voiconj, lewis} for a survey). Recall, any smooth, projective
variety $X$ over $\mb{C}$ admits the Hodge decomposition $H^r(X,\mb{C})=\oplus_i H^{i,r-i}(X, \mb{C})$. 
For $r = 2p$, elements of $H^{p,p}(X, \mb{Z}):=H^{2p}(X, \mb{Z}) \cap H^{p,p}(X, \mb{C})$ are called \emph{Hodge classes}. 
Denote by $Z^{p}(X)$ the free abelian group of algebraic cycles of $X$ of codimension $p$.
Recall, cycle class map  
\[c: Z^{p}(X) \to H^{2p}(X, \mb{Z})\] 
which sends an algebraic cycle to its cohomology class (see \cite[\S $11.1$]{v4}).
Denote by $H^{2p}_A(X, \mb{Z})$ the image of $c$, called the \emph{algebraic cohomology group}.
It is well-known that the cohomology class of any algebraic cycle is Hodge, i.e. $H^{2p}_A(X, \mathbb{Z}) \subset H^{p,p}(X,\mb{Z})$ (see \cite[Proposition $11.20$]{v4}).
The \emph{integral Hodge conjecture} predicts that every integral Hodge class comes from an algebraic cycle
i.e. $H^{2p}_A(X, \mathbb{Z})=H^{p,p}(X,\mb{C}) \cap H^{2p}(X,\mb{Z})$.
Although the conjecture is true for uniruled and Calabi Yau threefolds as well as 
cubic fourfolds (see \cite{voiconj}), the integral 
Hodge conjecture is false in general (see \cite{AH}).
Therefore, we instead consider the rational Hodge conjecture 
i.e. $H^{2p}_A(X, \mathbb{Q}) = H^{2p}(X, \mathbb{Q}) \cap H^{p,p}(X,\mb{C})$, 
where \[H^{2p}_A(X, \mathbb{Q}):=\Ima(c:Z^{p}(X) \otimes \mb{Q} \to H^{2p}(X, \mb{Q})).\]
 
One interesting case where the conjecture holds true is that of the Jacobian, $\mr{Jac}(C)$ of a \emph{very general}, smooth, projective curve $C$ (see \cite[\S $17.5$]{birk}). 
Using this, Balaji-King-Newstead in \cite{bala2} proved the conjecture for the moduli space $M_C(2,\mc{L})$ of rank $2$ semi-stable, 
locally free sheaves with determinant $\mc{L}$ over $C$, for an odd degree invertible 
sheaf $\mc{L}$ on $C$.
 It has also been shown for the moduli space of stable pairs over a smooth, projective curve in \cite{MOS}. However, nothing is known in the case the underlying curve is irreducible, nodal.The goal of this article is to prove the Hodge conjecture for the moduli space of stable rank $2$ bundles with odd degree determinant in the case when the underlying curve is very general, irreducible, nodal. 

 Let $X_0$ be a very general, irreducible, nodal curve with exactly one node, say $x_0$ and $\mc{L}_0$ an invertible sheaf on $X_0$ of odd degree. Here a \emph{very general nodal curve of genus} $g$ means that the normalization, together with the two inverse images of the node, is a \emph{very general} $2$-pointed curve of genus $g-1$. 
 Equivalently, a very general nodal curve lies outside countably many, proper closed subsets of the image of the clutching map from $\mc{M}_{g-1, 2}$ to $\ov{\mc{M}}_{g}$, where $\mc{M}_{g-1,2}$ denotes the moduli space of genus $g-1$ curves with $2$ marked points and $\ov{\mc{M}}_{g}$ is the moduli 
 space of stable curves of genus $g$ (see \cite[Chapter XII, \S $10$]{arb}). Given a torsion-free sheaf $\mc{E}$ on $X_0$, we say that $\mc{E}$ \emph{has determinant} $\mc{L}_0$ if there is an $\mathcal{O}_{X_0}$-morphism $\wedge^2(\mc{E})\to \mc{L}_0$ which is an isomorphism outside the node $x_0$.
 If $\mc{E}$ is locally free then this means $\wedge^2\mc{E}\cong \mc{L}_0$. 
 Using \cite[Theorem $2$]{sun1}, one can check that there exists a moduli space, denoted $U_{X_0}(2,\mc{L}_0)$,
 parameterizing rank $2$, semi-stable sheaves on $X_0$ with determinant $\mc{L}_0$. However, the moduli space $U_{X_0}(2,\mc{L}_0)$ is singular. We show:
 
 \begin{thm}
  For $X_0$ a very general, irreducible nodal curve, there exists a desingularization 
  $\mc{G}_0$ of $U_{X_0}(2,\mc{L}_0)$ (in the sense that $\mc{G}_0$ is non-singular and there is a proper birational morphism from $\mc{G}_0$ to $U_{X_0}(2,\mc{L}_0)$) such that the Hodge conjecture holds for $\mc{G}_0$. 
 \end{thm}

 See Theorem \ref{hod02} for a proof.

 One obstacle to simply generalizing the techniques used in the smooth curve case 
 is that an analogous description of the cohomology ring of $U_{X_0}(2,\mc{L}_0)$ is not available. More precisely, Balaji-King-Newstead 
 in \cite{bala2}  prove that there are Hodge classes 
 $\alpha \in H^2(M_C(2,\mc{L}), \mb{Z})$, $\beta \in H^4(M_C(2,\mc{L}), \mb{Z})$ and a surjective morphism
\[H^*(\mr{Jac}(C), \mb{Q}) \otimes \mb{Q}[\alpha, \beta] 
 \rightarrow H^*(M_C(2,\mc{L}), \mb{Q})\] inducing a surjective morphism  $\nu: H^*_A(\mr{Jac}(C), \mb{Q}) \otimes \mb{Q}[\alpha, \beta] 
 \to H^*_A(M_C(2,\mc{L}), \mb{Q}).$ Since the Hodge conjecture 
 holds for $\mr{Jac}(C)$ of a general smooth, projective curve $C$, 
 they are able to conclude the Hodge conjecture for $M_C(2,\mc{L})$. Unfortunately such a morphism does not exist if we 
 replace $M_C(2,\mc{L})$ by $U_{X_0}(2,\mc{L}_0)$. Moreover, since the Jacobian of a nodal curve is not projective, the statement of the 
 Hodge conjecture does not apply to the Jacobian of $X_0$. As a result the classical tools fail in this setup.
 
 Our strategy is as follows. We first embed the very general nodal curve $X_0$ as the central fiber of  a regular, flat family of projective curves $\pi:\mc{X} \to \Delta$ (here $\Delta$ denotes the unit disc), smooth over $\Delta^*:=\Delta\backslash \{0\}$ (see  \cite[Theorem B.$2$]{bake}) for the existence of such a family.
 Note that, the invertible sheaf $\mc{L}_0$ on $X_0$ lifts to a relative invertible sheaf $\mc{L}_{\mc{X}}$ over $\mc{X}$. By \cite{sun1} and \cite{tha}, there exists a relative Gieseker moduli space with fixed determinant over the family $\mc{X}$ given by a flat, projective morphism $\pi_2:\mc{G}(2,\mc{L}) \to \Delta$ such that for all $s \in \Delta^*$, $\mc{G}(2,\mc{L})_s:=\pi_2^{-1}(s)=M_{\mc{X}_s}(2,\mc{L}_s)$. The total space $\mc{G}(2,\mc{L})$ is regular and the central fiber $\pi_2^{-1}(0)$, denoted $\mc{G}_{X_0}(2,\mc{L}_0)$, 
 is a reduced simple normal crossings divisor of $\mc{G}(2,\mc{L})$ with two smooth, irreducible components such that one of them is a desingularization of $U_{X_0}(2,\mc{L}_0)$. 
We denote this desingularization by $\mc{G}_0$.
Since for all $s \neq 0$, $\mc{G}(2,\mc{L})_s = M_{\mc{X}_s}(2,\mc{L}_s)$ and we already know that the Hodge conjecture holds true for $ M_{\mc{X}_s}(2,\mc{L}_s)$ by \cite{bala2}, it is natural to compare the Hodge classes and the algebraic classes on $\mc{G}_{X_0}(2,\mc{L}_0)$ using variation of mixed Hodge structures. We prove that the Hodge conjecture holds for both of the smooth components of the central fibre and therefore for a desingularization of $U_{X_0}(2,\mc{L}_0)$. 


As a by-product we obtain the algebraic Poincar\'{e} polynomial of $\mc{G}_0$.
Recall by \cite[$(5.1)$]{bala2}, we have for any $s \in \Delta^*$, the algebraic Poincar\'{e} polynomial of $M_{\mc{X}_s}(2,\mc{L}_s)$, denoted 
 \[H(g,t):=\sum\limits_i H^i_A(M_{\mc{X}_s}(2,\mc{L}_s)) t^i=\frac{(1-t^g)(1-t^{g+1})(1-t^{g+2})}{(1-t)(1-t^2)(1-t^3)}, \] 
 where $\mc{X}_s:=\pi^{-1}(s)$ and $\mc{L}_s:=\mc{L}_{\mc{X}}|_{\mc{X}_s}$. Analogously, we prove (Theorems \ref{hod02} and \ref{cor11}): 
\begin{thm}\label{th11}
  The algebraic Poincar\'{e} polynomial for $\mc{G}_0$
 is given by \[P_A(\mc{G}_0):= \sum\limits_i H^i_A(\mc{G}_0) t^i = H(g-1,t)(t^2+t^4)+H(g,t).\] 
 \end{thm}

 We note that this article is part of a series of articles in which we study related but different questions pertaining to the moduli space of stable, rank $2$ sheaves on an irreducible nodal curve (see \cite{DK2, mumf, indpre, genpre}). The answers to these questions is well-known in the case when the underlying curve is smooth. Therefore we have often employed the theory of limit mixed Hodge structures to study the question for the nodal curve case using analogous results known for the case when the curve is smooth. However, the results in these articles are independent and overlap only in the background material.    
 
{\bf{Notation:}} Given any morphism $f:\mc{Y} \to S$ and a point $s \in S$, we denote by $\mc{Y}_s:=f^{-1}(s)$.
 The open unit disc is denoted by $\Delta$ and $\Delta^*:=\Delta\backslash \{0\}$ denotes the punctured disc.
 Unless mentioned otherwise, all cohomology groups are taken with $\mb{Q}$-coefficients.
 
 \emph{Acknowledgements} 
We thank Dr. S. Basu for numerous discussions. We are also grateful to the referee for his comments. 
The first author is funded by EPSRC grant number R/$162871-11-1$ and the second author
by the DFG, TRR $326$
Geometry and Arithmetic of Uniformized Structures, project number $444845124$.
Part of the work was done when the second author was funded by a PNPD fellowship from CAPES, Brazil.

\section{Preliminaries: Limit mixed Hodge structures}\label{sec3}
    We now briefly review the basics of limit mixed Hodge structures. Since the theory of limit mixed Hodge structures are used just as a tool, we only state definitions and results (without proof) relevant to our setup. For a detailed treatment of the subject see \cite{pet}.

\begin{note}
 Let $\rho:\mc{Y} \to \Delta$ be a flat, projective
 family of projective varieties, smooth over $\Delta^*$,
 $\rho':\mc{Y}_{\Delta^*} \to \Delta^*$ the restriction of $\rho$ to $\Delta^*$. 
 \end{note}
 
\begin{defi} 
 By Ehresmann's lemma (see \cite[Theorem $9.3$]{v4}), for all $i \ge 0$, $\mb{H}_{\mc{Y}_{\Delta^*}}^i:=R^i{\rho}'_{*}\mb{Z}$
 is a local system over $\Delta^*$ with fiber $H^i(\mc{Y}_t,\mb{Z})$, for $t \in \Delta^*$.
 One can associate to these local systems, the holomorphic vector bundles 
 $\mc{H}_{\mc{Y}_{\Delta^*}}^i:=\mb{H}_{\mc{Y}_{\Delta^*}}^i \otimes_{\mb{Z}} \mo_{\Delta^*}$
 called the \emph{Hodge bundle}.
 There exist holomorphic sub-bundles $F^p\mc{H}_{\mc{Y}_{\Delta^*}}^i \subset \mc{H}_{\mc{Y}_{\Delta^*}}^i$
defined by the condition: for any $t \in \Delta^*$, the fibers $\left(F^p\mc{H}_{\mc{Y}_{\Delta^*}}^i\right)_t \subset \left(\mc{H}_{\mc{Y}_{\Delta^*}}^i\right)_t$
can be identified respectively with 
$F^pH^i(\mc{Y}_t,\mb{C}) \subset H^i(\mc{Y}_t,\mb{C})$.
 where $F^p$ denotes the Hodge filtration (see \cite[\S $10.2.1$]{v4}).
  \end{defi}
  
 In order to define a mixed Hodge structure on the family $\rho:\mc{Y} \to \Delta$, the Hodge bundles and their holomorphic sub-bundles need to be extended to the entire disc.  By  \cite[Definition $11.4$]{pet}
 there exists a canonical extension $\ov{\mc{H}}_{\mc{Y}}^i$ of 
 ${\mc{H}}_{\mc{Y}_{\Delta^*}}^i$ to $\Delta$ .
 Note that, $\ov{\mc{H}}_{\mc{Y}}^i$ is locally-free over $\Delta$. Denote by $j:\Delta^* \to \Delta$ the inclusion morphism, then the Hodge filtration $F^p$ on $\Delta^*$ is extended to $\Delta$ by setting 
 $F^p\ov{\mc{H}}_{\mc{Y}}^i:= j_*\left(F^p\mc{H}_{\mc{Y}_{\Delta^*}}^i\right) \cap  \ov{\mc{H}}_{\mc{Y}}^i$.
 Note that, $F^p\ov{\mc{H}}_{\mc{Y}}^i$ is the locally-free
 sub-sheaf of $\ov{\mc{H}}_{\mc{Y}}^i$ which extends $F^p\mc{H}_{\mc{Y}_{\Delta^*}}^i$.
 
 \begin{defi}\label{defiunc}
 Consider the universal cover $\mf{h} \to \Delta^*$ of the punctured unit disc. 
 Denote by $e:\mf{h} \to \Delta^* \xrightarrow{j} \Delta$ the composed morphism and define by  
 \[\mc{Y}_\infty:=\mc{Y} \times_{\Delta} \mf{h}\] the base change of the family $\mc{Y}$ over $\Delta$ to $\mf{h}$, by the morphism $e$.
   
   By \cite[XI-$8$]{pet}, for a choice of the parameter $t$ on $\Delta$, the central fiber of the canonical extension $\ov{\mc{H}}_{\mc{Y}}^i$ can be identified with the cohomology group $H^i(\mc{Y}_{\infty},\mb{C})$:  
 \begin{equation}\label{tor23}
  g^i_{_t}:H^i(\mc{Y}_{\infty},\mb{C}) \xrightarrow{\sim} \left(\ov{\mc{H}}_{\mc{Y}}^i\right)_0.
 \end{equation}
  As a consequence, there exist Hodge filtrations on $H^i(\mc{Y}_{\infty},\mb{C})$ defined by
\[F^pH^i(\mc{Y}_{\infty},\mb{C}):=(g_{_t}^i)^{-1}\left(F^p\ov{\mc{H}}_{\mc{Y}}^i\right)_0.\]
\end{defi}

To define a weight filtration on $H^i(\mc{Y}_{\infty},\mb{Z})$, we use the local monodromy transformations. 

\begin{defi}
For any $s \in \Delta^*$ and $i \ge 0$, denote by 
\[T_{s,i}: H^i(\mc{Y}_s,\mb{Z}) \to H^i(\mc{Y}_s,\mb{Z}) \, \mbox{ and }\, T_{s,i}^{\mb{Q}}: H^i(\mc{Y}_s) \to H^i(\mc{Y}_s)\]
the \emph{local monodromy transformations} associated to the local system $\mb{H}_{\mc{Y}_{\Delta^*}}^i$,
   defined by parallel transport along a counterclockwise loop about $0 \in \Delta$ (see \cite[\S $11.1.1$]{pet}).
   By Ehresmann's lemma, the cohomology group $H^i(\mc{Y}_{\infty})$ is (canonically)
   isomorphic to $H^i(\mc{Y}_s)$, depending only on the choice of 
   $s' \in \mf{h}$ with $e(s')=s$, 
   induced by the inclusion of $\mc{Y}_{s'}$ into 
   $\mc{Y}_\infty$. Then using the translation $s' \mapsto s'+1$, the automorphism
   $T_{s,i}^{\mb{Q}}$ induces a $\mb{Q}$-automorphism 
   \begin{equation}\label{int01}
 T_i: H^i(\mc{Y}_{\infty}) \to H^i(\mc{Y}_{\infty}).
\end{equation}
 See \cite[\S II.$2.4$]{kuli} or \cite[Theorem II.$1.17$]{deli2} for more details.
\end{defi}

We can now define a mixed Hodge structure on $H^i(\mc{Y}_{\infty},\mb{Z})$. 

  \begin{defi}\label{tor25}\label{defimhs}
  Let $N_i$ be the logarithm of the monodromy operator $T_i$.
  By \cite[Lemma-Definition $11.9$]{pet}, there exists an unique increasing \emph{monodromy weight filtration} $W_\bullet$ on $H^i(\mc{Y}_\infty)$ such that,
 \begin{enumerate}
  \item  for $j \ge 2$, $N_i(W_jH^i(\mc{Y}_\infty)) \subset W_{j-2}H^i(\mc{Y}_\infty)$ and
  \item the map $N_i^l: \mr{Gr}^W_{i+l} H^i(\mc{Y}_\infty) \to \mr{Gr}^W_{i-l} H^i(\mc{Y}_\infty)$ 
  is an isomorphism for all $l \ge 0$.
   \end{enumerate}
  By \cite[Theorem $6.16$]{schvar} the induced filtrations on  $H^i(\mc{Y}_{\infty},\mb{C})$ define a mixed Hodge structure $(H^i(\mc{Y}_{\infty},\mb{Z}),W_\bullet,F^\bullet)$, called the \emph{limit mixed Hodge structure} on $H^i(\mc{Y}_{\infty},\mb{Z})$.
  \end{defi}

  \begin{defi} 
 Recall, for any $t \in \Delta^*$, there exist natural specialization
 morphisms:
 \[\mr{sp}_i:H^i(\mc{Y}_0) \to H^i(\mc{Y}_t) \]
 obtained by composing the natural inclusion of the special fiber
 $\mc{Y}_t$  into 
 $\mc{Y}$ with the retraction map
 to the central fiber $\mc{Y}_0$. Unfortunately, the resulting specialization maps are not morphism of mixed Hodge structures. However, if one identifies
 $H^i(\mc{Y}_t)$ with $H^i(\mc{Y}_\infty)$ as mentioned above, then the resulting 
 (modified) specialization morphisms are morphisms of mixed Hodge structures (see \cite[Theorem $11.29$]{pet}).
  Furthermore, by the local invariant cycle theorem \cite[Theorem $11.43$]{pet}, we have the following exact sequence:
 \begin{align}
 &H^i(\mc{Y}_0) \xrightarrow{\mr{sp}_i} H^i(\mc{Y}_{\infty}) \xrightarrow{T_{\mc{Y}} - \mr{Id}} H^i(\mc{Y}_{\infty}) 
 \end{align}
 where $\mr{sp}_i$ denotes the (modified) specialization morphism which is a morphism of mixed Hodge structures
as discussed above.
\end{defi}

Suppose that $\mc{Y}_0$ is a simple, 
 normal crossings divisor and an union of two smooth, 
 irreducible components, say $Y_1, Y_2$. 
 Since $Y_1 \cap Y_2, Y_1, Y_2$ are non-singular (hence has pure Hodge structure), 
  the  Mayer-Vietoris sequence associated to $\mc{Y}_0$, 
  gives rise to the exact sequence:
  \begin{equation}\label{eq:viet}
   0 \to \mr{Gr}^W_i H^i(\mc{Y}_0) \to H^i(Y_1) \oplus H^i(Y_2)
   \xrightarrow{r_1-r_2} H^i(Y_1 \cap Y_2), 
  \end{equation}
  where $r_j$ is the restriction morphism from $H^i(Y_j)$ to $H^i(Y_1 \cap Y_2)$
  for $j=1,2$.
  Note that the Gysin morphism from $H^{i-2}(Y_1 \cap Y_2)(-1)$
  to $H^i(Y_1) \oplus H^i(Y_2)$ composed with $(r_1-r_2)$ is the zero map
  (see \cite[Proposition B.$30$]{pet}).
 Hence, $H^i(Y_1 \cap Y_2)$ factors through $\mr{Gr}^W_i H^i(\mc{Y}_0)$.
 Denote by 
  \begin{equation}\label{eq:ntor03}
   f_i: H^{i-2}(Y_1 \cap Y_2)(-1) \to \mr{Gr}^W_i H^i(\mc{Y}_0) \hookrightarrow H^i(\mc{Y}_0)
  \end{equation}
  the composed morphism.
  We recall the following useful result which we use repeatedly.
 
 \begin{cor}\label{ntor03}
   Suppose that the central fiber $\mc{Y}_0$ is a simple, 
 normal crossings divisor and an union of two smooth, 
 irreducible components, say $Y_1, Y_2$. Then,  
 $W_{i+1} H^i(\mc{Y}_\infty) = H^i(\mc{Y}_\infty)$ and 
  we have the following exact sequence of mixed Hodge structures:
  \begin{equation}\label{ntor02}
 H^{i-2}(Y_1 \cap Y_2)(-1) \xrightarrow{f_i} H^i(\mc{Y}_0) \xrightarrow{\mr{sp}_i} H^i(\mc{Y}_\infty) \xrightarrow{g_i} \mr{Gr}^W_{i+1} H^i(\mc{Y}_\infty) \to 0,
  \end{equation}
where $f_i$ is the morphism \eqref{eq:ntor03}
and $g_i$ is the natural quotient by $W_{i} H^i(\mc{Y}_\infty)$.
 \end{cor}
 
 \begin{proof}
  See \cite[Corollary $2.4$]{mumf} for a proof.
 \end{proof}

   \section{Limit mixed Hodge structures of the relative Gieseker moduli space}\label{sec2}
 
 In this section we recall the basic definitions and results on the relative Gieseker moduli space,
 the associated monodromy action and extend the exact sequence \eqref{ntor02} in this setup.
 
  \subsection{Notation}\label{tor33}
 Let $X_0$ be an irreducible nodal curve of genus $g \ge 2$, with exactly
 one node, say at $x_0$. Denote by $\pi_0:\widetilde{X}_0 \to X_0$ the normalization map.
 There exists a flat, projective family of curves
  $\pi_1:\mc{X} \to \Delta$ smooth over 
   $\Delta^*$ and central fiber isomorphic to $X_0$ (see \cite[Theorem B$.2$]{bake}).
   Moreover, $\mc{X}$ is a regular variety.
   Fix an invertible sheaf $\mc{L}$ on $\mc{X}$ such that $\mc{L}|_{\mc{X}_s}$ is of 
   odd degree, say $d$ for all $s \in \Delta$. Set $\mc{L}_0:=\mc{L}|_{X_0}$, 
   the restriction of $\mc{L}$ to the central fiber.
   Denote by $\widetilde{\mc{L}}_0:=\pi_0^*(\mc{L}_0)$.

 \subsection{Relative Gieseker moduli space}\label{subsec1}
 Recall, that a curve $X_k$ \emph{is semi-stably equivalent to} $X_0$ if it is the union of the normalization $\widetilde{X}_0$ and a chain of rational curves
 of length $k$. See \cite[Definition-Notation 2]{nagsesh} for the precise definition. 
 We say that a family of curves $f:\mc{X}_S \to S$ \emph{is semi-stably equivalent to the family} $\pi_1$ above, if 
 \begin{enumerate}
  \item there is a morphism from $\mc{X}_S$ to $\mc{X}$ inducing a morphism $h$ from $S$ to $\Delta$ such that
  the resulting diagram is commutative,
  \item for any point $s$ mapping to $0 \in \Delta$, the fiber $f^{-1}(s)$ is semi-stably equivalent to $X_0$. For other points $s \in S$, 
  the fiber $f^{-1}(s)$ is isomorphic to the fiber of $\pi_1$ over $h^{-1}(s)$.
 \end{enumerate}
 There exists a relative moduli space, called the \emph{relative Gieseker moduli space}, denoted $\mc{G}(2,\mc{L})$ which parametrizes certain rank $2$, determinant $\mc{L}$
 semi-stable sheaves defined over families of curves semi-stably equivalent to the family $\pi_1$. See \cite[\S $3$]{sun1} or \cite[\S $6$]{tha}
 for the precise definitions.
 By \cite[Theorem $2$]{sun1}, there exists a flat, projective morphism $\pi_2:\mc{G}(2,\mc{L}) \to \Delta$ such that for all $s \in \Delta^*$, $\mc{G}(2,\mc{L})_s:=\pi_2^{-1}(s)=M_{\mc{X}_s}(2,\mc{L}_s)$ 
 (see also \cite[\S $5$ and $6$]{abe}).
 Moreover, $\mc{G}(2,\mc{L})$ is regular and the central fiber $\pi_2^{-1}(0)$, denoted $\mc{G}_{X_0}(2,\mc{L}_0)$, 
 is a reduced simple normal crossings divisor of $\mc{G}(2,\mc{L})$ (see \cite[\S $6$]{tha}).

   Denote by $M_{\widetilde{X}_0}(2,\widetilde{\mc{L}}_0)$ the fine moduli space of semi-stable sheaves of rank $2$ and with 
  determinant $\widetilde{\mc{L}}_0$ over $\widetilde{X}_0$ (see \cite[Theorem $4.3.7$ and $4.6.6$]{huy}).
  By \cite[$(6.2)$]{tha}, $\mc{G}_{X_0}(2,\mc{L}_0)$ can be written as the union of two irreducible components, say $\mc{G}_0$ and $\mc{G}_1$, where 
  $\mc{G}_1$ (resp. $\mc{G}_0 \cap \mc{G}_1$) is isomorphic to a $\p3$ (resp. $\mb{P}^1 \times \mb{P}^1$)-bundle
  over $M_{\widetilde{X}_0}(2,\widetilde{\mc{L}}_0)$. Note that, there is a proper morphism $\theta:\mc{G}_{X_0}(2,\mc{L}_0) \to U_{X_0}(2,\mc{L}_0)$
     with the irreducible component $\mc{G}_0$ mapping surjectively to $U_{X_0}(2,\mc{L}_0)$ (use \cite[Theorem $3.7$]{sun1} and \cite[\S $6$]{tha}).
     In fact, the restriction of $\theta$ to $\mc{G}_0$ is a birational morphism. Since $\mc{G}_0$ is non-singular, it 
     is a desingularization of $U_{X_0}(2,\mc{L}_0)$ (see \cite{tha, sun1}). Moreover, there exists an $\ov{\mr{SL}}_2$-bundle over $M_{\widetilde{X}_0}(2,\widetilde{\mc{L}}_0)$, 
  denoted $P_0$, and closed subschemes $Z \subset P_0$, $Z' \subset \mc{G}_0$ 
  such that $P_0 \backslash Z \cong \mc{G}_0 \backslash Z'$ and 
  $\emptyset = Z' \cap \mc{G}_1$ (see \cite[p. $27$]{tha}),
   where $\ov{\mr{SL}}_2$ is the \emph{wonderful compactification} of $\mr{SL}_2$ defined as 
    \[\ov{\mr{SL}}_2:= \{[M,\lambda] \in \mb{P}(\mr{End}(\mb{C}^2) \oplus \mb{C})| \det(M)=\lambda^2\}\] 
    (see \cite[Definition $3.3.1$]{pezz} for a general definition of wonderful compactification).

 \subsection{Monodromy action on the relative Gieseker moduli space}\label{sec:mon}
An important step in proving the Hodge conjecture for $\mc{G}_0$ is to study the 
 monodromy action on $H^{2i}(\mc{G}(2,\mc{L})_s)$
 for $s \in \Delta^*$. For this purpose, we need to consider the relative version of the construction of the Mumford-Newstead isomorphism \cite{mumn}.
  Denote by  \[\mc{W}:=\mc{X}_{\Delta^*} \times_{\Delta^*} \mc{G}(2,\mc{L})_{\Delta^*}\, \mbox{ and }\, \pi_3: \mc{W} \to \Delta^*\] the natural morphism. We set $\mc{W}_t:=\pi_3^{-1}(t) \cong \mc{X}_t \times M_{\mc{X}_t}(2,\mc{L}_t)$ for all $t \in \Delta^*$.
  There exists a (relative) universal bundle $\mc{U}$ over $\mc{W}$ 
  associated to the (relative) moduli space $\mc{G}(2,\mc{L})_{\Delta^*}$. Note that $\mc{U}$ is only well-defined up to pull-back by a line bundle on $\mc{G}(2,\mc{L})_{\Delta^*}$. For each $t \in \Delta^*$,
  $\mc{U}|_{\mc{W}_t}$ is a universal bundle over $\mc{X}_t \times M_{\mc{X}_t}(2,\mc{L}_t)$ associated to fine moduli space $M_{\mc{X}_t}(2,\mc{L}_t)$
  (use \cite[Theorem $9.1.1$]{pan}).
  Let $\mb{H}^4_{\mc{W}}:=R^4 \pi_{3_*} \mb{Z}_{\mc{W}}$ be the local system associated to $\mc{W}$.
  By K\"{u}nneth decomposition, we have 
  \begin{equation}
   \mb{H}^4_{\mc{W}}= \bigoplus\limits_i \left(\mb{H}^i_{\mc{X}_{\Delta^*}} \otimes \mb{H}^{4-i}_{\mc{G}(2,\mc{L})_{\Delta^*}}\right).
  \end{equation}
  Denote by $c_2(\mc{U})^{1,3} \in \Gamma\left(\mb{H}^1_{\mc{X}_{\Delta^*}} \otimes \mb{H}^{3}_{\mc{G}(2,\mc{L})_{\Delta^*}}\right)$ 
  the image of the second Chern class $c_2(\mc{U}) \in \Gamma(\mb{H}^4_{\mc{W}})$ under the natural projection from 
  $\mb{H}^4_{\mc{W}}$ to $\mb{H}^1_{\mc{X}_{\Delta^*}} \otimes \mb{H}^{3}_{\mc{G}(2,\mc{L})_{\Delta^*}}$.
  Using Poincar\'{e} duality applied to the local system $\mb{H}^1_{\mc{X}_{\Delta^*}}$, we have
\begin{equation}\label{ner08}
 \mb{H}^1_{\mc{X}_{\Delta^*}} \otimes \mb{H}^{3}_{\mc{G}(2,\mc{L})_{\Delta^*}} \stackrel{\mr{PD}}{\cong}\left(\mb{H}^1_{\mc{X}_{\Delta^*}}\right)^\vee \otimes \mb{H}^{3}_{\mc{G}(2,\mc{L})_{\Delta^*}} \cong \Hc\left(\mb{H}^1_{\mc{X}_{\Delta^*}}, \mb{H}^{3}_{\mc{G}(2,\mc{L})_{\Delta^*}}\right).
\end{equation}
  Therefore, $c_2(\mc{U})^{1,3}$ induces a homomorphism
  $\Phi_{\Delta^*}: \mb{H}^1_{\mc{X}_{\Delta^*}} \to \mb{H}^{3}_{\mc{G}(2,\mc{L})_{\Delta^*}}$.
  Note that, the restriction of the universal bundle $\mc{U}$ to the fiber over any point of 
  the moduli space $\mc{G}(2,\mc{L})_{\Delta^*}$ is uniquely determined by the point (and 
  not by the choice of the universal bundle). Hence, the value of the section $c_2(\mc{U})$
  at each point of $\Delta^*$ is independent of the choice of $\mc{U}$. Since a (local) section of 
  a local system is uniquely determined by its value at a point, the section $c_2(\mc{U})$
  does not depend on the choice of the universal bundle. As a result, the homomorphism $\Phi_{\Delta^*}$ is independent of the choice of $\mc{U}$.
  By \cite[Lemma $1$ and Proposition $1$]{mumn}, we conclude that the homomorphism $\Phi_{\Delta^*}$
  is an isomorphism such that the induced 
  isomorphism on the associated vector bundles:
  \[\Phi_{\Delta^*}:\mc{H}^1_{\mc{X}_{\Delta^*}} \xrightarrow{\sim} \mc{H}^{3}_{\mc{G}(2,\mc{L})_{\Delta^*}} \mbox{ satisfies } \Phi_{\Delta^*}(F^p\mc{H}^1_{\mc{X}_{\Delta^*}})= F^{p+1}\mc{H}^{3}_{\mc{G}(2,\mc{L})_{\Delta^*}} \mbox{ for all }p \ge 0.\]
  Denote by 
   \begin{equation}\label{eq11}
    \widetilde{\Phi}_s: H^1(\mc{X}_s, \mb{Z}) \xrightarrow{\sim} H^3(\mc{G}(2,\mc{L})_s, \mb{Z})
   \end{equation}
   the restriction of $\Phi_{\Delta^*}$ to the point $s \in \Delta^*$. 
    Since $c_2(\mc{U})^{1,3}$ is a (single-valued) global section of $\mb{H}^1_{\mc{X}_{\Delta^*}} \otimes \mb{H}^{3}_{\mc{G}(2,\mc{L})_{\Delta^*}}$,
   we have 
   \[\widetilde{\Phi}_s \in \mr{Hom}(H^1(\mc{X}_s, \mb{Z}), H^3(\mc{G}(2,\mc{L})_s, \mb{Z})) \cong H^1(\mc{X}_s, \mb{Z})^\vee \otimes H^3(\mc{G}(2,\mc{L})_s, \mb{Z}) \stackrel{\mr{P.D.}}{\cong}\]\[\stackrel{\mr{P.D.}}{\cong}
           H^1(\mc{X}_s, \mb{Z}) \otimes H^3(\mc{G}(2,\mc{L})_s, \mb{Z}) 
           \]
is monodromy invariant i.e., for all $s \in \Delta^*$, the following diagram is commutative:
   \begin{equation}\label{ner03}
    \begin{diagram}
     H^1(\mc{X}_s, \mb{Z}) &\rTo^{\widetilde{\Phi}_s}_{\sim}& H^3(\mc{G}(2,\mc{L})_s, \mb{Z}) \\
  \dTo^{T_{\mc{X}_s}}_{\cong}&\circlearrowleft&\dTo_{T_{\mc{G}(2,\mc{L})_s}}^{\cong}\\
    H^1(\mc{X}_s, \mb{Z}) &\rTo^{\widetilde{\Phi}_s}_{\sim}&H^3(\mc{G}(2,\mc{L})_s, \mb{Z})  
       \end{diagram}
   \end{equation}
 where $T_{\mc{X}_s}$ and $T_{\mc{G}(2,\mc{L})_s}$ are the monodromy transformations on $H^1(\mc{X}_s, \mb{Z})$
 and $H^3(\mc{G}(2,\mc{L})_s, \mb{Z})$, respectively.

 \subsection{Limit Mumford-Newstead isomorphism}\label{sec:limmum} 
 Let $\ov{\mc{H}}^1_{\mc{X}}$ and $\ov{\mc{H}}^{3}_{\mc{G}(2,\mc{L})}$ be the canonical extensions of $\mc{H}^1_{\mc{X}_{\Delta^*}}$
  and $\mc{H}^{3}_{\mc{G}(2,\mc{L})_{\Delta^*}}$, respectively. Then, the morphism $\Phi_{\Delta^*}$ extends to the entire disc (by canonicity of the extensions):
  \[\widetilde{\Phi}:\ov{\mc{H}}_{\mc{X}}^1 \xrightarrow{\sim} \ov{\mc{H}}^3_{\mc{G}(2,\mc{L})}.\]
  Using the identification \eqref{tor23} and restricting $\widetilde{\Phi}$ to the central fiber, we have an isomorphism:
  \begin{equation}\label{ner11}
   \widetilde{\Phi}_0 : H^1(\mc{X}_\infty) \xrightarrow{\sim} H^3(\mc{G}(2,\mc{L})_\infty),
  \end{equation}
  where $\mc{X}_\infty$ (resp. $\mc{G}(2,\mc{L})_{\infty}$) is the base change of $\mc{X}$ (resp. 
  $\mc{G}(2, \mc{L})$) to the upper half plane $\mf{h}$ as described in Definition \ref{defiunc}.
   Recall, $\widetilde{\Phi}_0$ is an isomorphism of mixed Hodge structures:
   
  \begin{thm}\label{ner01}
   For the extended morphism $\widetilde{\Phi}$, we have 
   $\widetilde{\Phi}(F^p\ov{\mc{H}}_{\mc{X}}^1)= F^{p+1}\ov{\mc{H}}^3_{\mc{G}(2,\mc{L})}$  for  $p=0,1$  and  
   $\widetilde{\Phi}(\ov{\mb{H}}_{\mc{X}}^1)=\ov{\mb{H}}^3_{\mc{G}(2,\mc{L})}$.
   Moreover, $\widetilde{\Phi}_0(W_iH^1(\mc{X}_\infty))=W_{i+2}H^3(\mc{G}(2,\mc{L})_\infty)$
   for all $i \ge 0$.
  \end{thm}

  \begin{proof}
   See \cite[Proposition $4.1$]{indpre} for a proof of the statement.
 \end{proof}

 \subsection{Leray-Hirsch cohomology decomposition}\label{note:ner01}
 We now write the cohomology groups of $\mc{G}_1$ and $\mc{G}_0 \cap \mc{G}_1$ in terms of that of $M_{\widetilde{X}_0}(2,\widetilde{\mc{L}}_0)$,
 using the cohomology computations of the fibers of $\mc{G}_1$ and $\mc{G}_0 \cap \mc{G}_1$ over $M_{\widetilde{X}_0}(2,\widetilde{\mc{L}}_0)$. 
 
 \begin{note}\label{note:inc}
  Let $\xi_0$ be a generator of $H^2(\mb{P}^1)$, $\pr_i$ the natural 
projections from $\mb{P}^1 \times \mb{P}^1$ to $\mb{P}^1$ and $\xi_i:=\pr_i^*(\xi_0)$.
Using the Künneth decomposition, we have $H^{2i+1}(\mb{P}^1 \times \mb{P}^1)=0$ for $i \ge 0$,
\[H^0(\mb{P}^1 \times \mb{P}^1)= \mb{Q},\,     
H^2(\mb{P}^1 \times \mb{P}^1)= \mb{Q}\xi_1 \oplus \mb{Q}\xi_2\mbox{ and } H^4(\mb{P}^1 \times \mb{P}^1)=\mb{Q} \xi_1.\xi_2.\]
In other words, the cohomology ring 
$H^*(\mb{P}^1 \times \mb{P}^1)=\mb{Q}[\xi_1,\xi_2]/(\xi_1^2,\xi_2^2)$. Recall, the cohomology ring 
$H^*(\mb{P}^3)=\mb{Q}[\xi'']/((\xi'')^4)$, where $\xi'' \in H^2(\mb{P}^3, \mb{Z})$ is a generator.
Moreover, $\mb{P}^1 \times \mb{P}^1$ can be embedded (via Segre embedding)
as a smooth quadric surface in $\p3$ and under the induced pull-back morphism 
of (graded) cohomology rings:
\[\mb{Q}[\xi'']/((\xi'')^4)=H^*(\p3) \to H^*(\mb{P}^1 \times \mb{P}^1) = \mb{Q}[\xi_1,\xi_2]/(\xi_1^2,\xi_2^2),\]
$\xi''$ is mapped to $\xi_1+\xi_2$.
 Denote by 
  \[i_0: \mc{G}_0 \cap \mc{G}_1 \hookrightarrow \mc{G}_0
  \mbox{ and } i_1:\mc{G}_0 \cap \mc{G}_1 \hookrightarrow \mc{G}_1\] the natural inclusions.
 Let
$\rho_1: \mc{G}_0 \cap \mc{G}_1 \longrightarrow M_{\widetilde{X}_0}(2,\widetilde{\mc{L}}_0)$ and 
 $\rho_2: \mc{G}_1 \longrightarrow M_{\widetilde{X}_0}(2,\widetilde{\mc{L}}_0)$
 be the natural bundle morphisms. Note that the restriction morphism $r_j: H^{2j}(\p3) \to H^{2j}(\mb{P}^1 \times \mb{P}^1)$ 
 is an isomorphism (Lefschetz hyperplane 
 section theorem), for any $j \not= \{1,3\}$.
 \end{note}

  By the Deligne-Blanchard theorem \cite{delibla} (the Leray spectral sequence
degenerates at $E_2$ for smooth families), we have 
$H^i(\mc{G}_0 \cap \mc{G}_1) \cong  \oplus_{_j} H^{i-j}(R^j \rho_{1,*}\mb{Q})$
 and $H^i(\mc{G}_1) \cong  \oplus_{_j} H^{i-j}(R^j \rho_{2,*}\mb{Q})$.
Since $M_{\widetilde{X}_0}(2,\widetilde{\mc{L}}_0)$ is smooth and simply connected 
(as it is rationally connected \cite[Proposition $2.3.7$ and Remark $2.3.8$]{ind} which implies 
simply connected \cite[Corollary $4.29$]{deba}), the local systems
$R^j \rho_{1,*} \mb{Q}$ and $R^j \rho_{2,*} \mb{Q}$ are trivial.
Therefore, for any $y \in M_{\widetilde{X}_0}(2,\widetilde{\mc{L}}_0)$, the natural morphisms
{\small \[H^i(\mc{G}_0 \cap \mc{G}_1) \twoheadrightarrow H^0(R^i \rho_{1,*} \mb{Q}) \to H^i((\mc{G}_0 \cap \mc{G}_1)_y)
 \mbox{ and }  H^i(\mc{G}_1) \twoheadrightarrow H^0(R^i \rho_{2,*} \mb{Q}) \to H^i(\mc{G}_{1,y})
\]}
are surjective. 
By the Leray-Hirsch theorem (see \cite[Theorem $7.33$]{v4}), we then have:
{\small \begin{align}
 H^i(\mc{G}_0 \cap \mc{G}_1) & \cong \bigoplus_{j \ge 0} (H^j((\mc{G}_0 \cap \mc{G}_1)_y) \otimes H^{i-j}(M_{\widetilde{X}_0}(2,\widetilde{\mc{L}}_0)))\nonumber \\ \label{ntor04}
& \cong H^i(M_{\widetilde{X}_0}(2,\widetilde{\mc{L}}_0)) \oplus H^{i-2}(M_{\widetilde{X}_0}(2,\widetilde{\mc{L}}_0)) \otimes (\mb{Q}\xi_1 \oplus \mb{Q}\xi_2) \oplus H^{i-4}(M_{\widetilde{X}_0}(2,\widetilde{\mc{L}}_0)) \xi_1\xi_2, \\ 
H^i(\mc{G}_1) & \cong \bigoplus_{j \ge 0} (H^j(\mc{G}_{1,y}) \otimes H^{i-j}(M_{\widetilde{X}_0}(2,\widetilde{\mc{L}}_0)))
 \cong \bigoplus\limits_{j = 0}^{3} H^{i-2j}(M_{\widetilde{X}_0}(2,\widetilde{\mc{L}}_0)) \otimes (\xi'')^j \label{ntor05} 
\end{align}}
  Using this one can check that:
 
 \begin{prop}\label{ner15}
  The following holds true:
  \[\ker((i_{0,*},i_{1,*}):H^{i-2}(\mc{G}_0 \cap \mc{G}_1) \to H^i(\mc{G}_0) \oplus H^i(\mc{G}_1)) \cong H^{i-4}(M_{\widetilde{X}_0}(2,\widetilde{\mc{L}}_0))(\xi_1 \oplus -\xi_2).\]
   \end{prop}
 
 \begin{proof}
  See \cite[Theorem $4.2$]{mumf} for proof of the statement.
 \end{proof}

 Let $\mc{G}(2,\mc{L})_\infty$ denote the base change of the family $\mc{G}(2,\mc{L})$
 (defined in \S \ref{subsec1}) over $\Delta$ to the universal cover $\mf{h}$ by the morphism  $e:\mf{h} \to \Delta^* \xrightarrow{j} \Delta$. Using the definitions and results from  \S \ref{sec3}, we can 
 equip $H^{i}(\mc{G}(2,\mc{L})_\infty)$ with a limit mixed Hodge structure 
 and obtain a specialization morphism:
  \[\mr{sp}_i: H^i(\mc{G}_{X_0}(2,\mc{L}_0)) \to H^i(\mc{G}(2,\mc{L})_\infty)\]
which is a morphism of mixed Hodge structures for all $i$.
 Replacing $Y_1$ (resp. $Y_2$) in \S \ref{sec3} by $\mc{G}_0$ (resp. $\mc{G}_1$) from \S \ref{subsec1}, we 
 have by Proposition \ref{ner15}
  that the kernel of the morphism $f_i$ is isomorphic
  to $H^{i-4}(M_{\widetilde{X}_0}(2,\widetilde{\mc{L}}_0))(-2)$ i.e., 
  we have an exact sequence of mixed Hodge structures (see Corollary \ref{ntor02}):
    \begin{eqnarray}\label{ntor09}
 &0 \to H^{i-4}(M_{\widetilde{X}_0}(2,\widetilde{\mc{L}}_0))(-2) \xrightarrow{h_i} H^{i-2}(\mc{G}_0 \cap \mc{G}_1)(-1) \xrightarrow{f_i} H^i(\mc{G}_{X_0}(2,\mc{L}_0)) \xrightarrow{\mr{sp}_i} \nonumber\\
 &\xrightarrow{\mr{sp}_i} H^i(\mc{G}(2,\mc{L})_\infty) \xrightarrow{g_i} \mr{Gr}^W_{i+1} H^i(\mc{G}(2,\mc{L})_\infty) \to 0.
  \end{eqnarray}

 \section{Limit Hodge conjecture for relative Gieseker moduli space}
Recall, that for a very general smooth, projective curve $Y$, $H^{2p}_{\mr{Hdg}}(\mr{Jac}(Y)) \cong \mb{Q}$
for every $p$ (see \cite[Theorem $17.5.1$]{birk}. It is natural to ask, given a degenerating family of 
smooth, projective curves, is the sub-space of $p$-th Hodge classes in the limit mixed Hodge 
structure associated to the corresponding family of Jacobians, one dimensional for each $p$?
Apriori, this is false as can be seen from the theory of Noether-Lefschetz loci i.e., very generally, the 
rank of the Hodge lattice jumps. However, in this section we prove that under the genericity assumption of the 
nodal curve $X_0$, we have a positive answer (see Theorem \ref{thm:lim}). This result will later 
play an important role in the proof of Theorem \ref{hod02}.

 \subsection{Hodge conjecture in the smooth case}\label{recsmcase}
  We now recall from \cite{bala2}
  the proof of the Hodge conjecture for the moduli space of rank $2$ vector bundles with fixed determinant on a very general, 
  smooth, projective curve. We give a brief sketch for the sake of completion.  

  Let $C$ be a smooth, projective curve and $\mc{L}$ an invertible sheaf on $C$ of odd degree.
  Denote by $M_C(2,\mc{L})$ the moduli of semi-stable, rank $2$ sheaves on $C$ with determinant $\mc{L}$.
 By \cite[Proposition $1$]{mumn}, there exists an isomorphism of pure Hodge structures:
   \[\Phi:H^1(C, \mb{Z})(-1) \to H^3(M_{C}(2,\mc{L}), \mb{Z}).\]
   By \cite[Theorem $1$]{new1}, there exists 
 $\alpha \in H^2(M_{C}(2,\mc{L}), \mb{Z})$
 and $\beta \in H^4(M_{C}(2,\mc{L}), \mb{Z})$ such that the cohomology ring $H^*(M_{C}(2,\mc{L}))$ is generated by $\alpha, \beta$
 and $H^3(M_C(2,\mc{L}))$. 
 Now, there is a natural isomorphism of pure Hodge structures 
 \begin{equation}\label{eq:jac}
  H^1(\mr{Jac}(C),\mb{Z}) \xrightarrow{\sim} H^1(C,\mb{Z})^\vee \xrightarrow{\sim} H^1(C,\mb{Z})
 \end{equation}
 where the first isomorphism comes from the cohomology of Jacobians (see \cite[p. $169$]{v4}) and 
 the second isomorphism is the Poincar\'{e} duality. Combining with the isomorphism $\Phi$ above, we get
 an isomorphism of pure Hodge structures:
 \[H^1(\mr{Jac}(C), \mb{Z})(-1) \to H^1(C, \mb{Z})(-1) \xrightarrow{\Phi} H^3(M_C(2,\mc{L}), \mb{Z}).\]
 Since cup-product is a morphism of pure Hodge structures,
 this isomorphism induces a surjective (graded) ring homomorphism which is a morphism of pure Hodge 
 structures on each graded piece ($\alpha$ and $\beta$ are given weights $2$ and $4$, respectively):
 \begin{equation}\label{eq:gies}
   \delta_C:\, \mb{Q}[\alpha, \beta] \otimes (\oplus_i\, \, H^i(\mr{Jac}(C))(-i)) \to H^*(M_{C}(2,\mc{L}))
 \end{equation}
(see \cite[$(1.2)$]{bala2}). As $\alpha$ and $\beta$ are linear combination of Chern classes of 
the universal bundle associated to $M_C(2,\mc{L})$, they are algebraic cohomology classes 
(see \cite[p. $338$]{new1}). 
 Then, by \cite[Theorem $2$]{bala2} the isomorphism $\delta_C$ induces a surjective ring homomorphism:
 \[\delta_C:\,   \mb{Q}[\alpha, \beta] \otimes H^*_A(\mr{Jac}(C)) \to H^*_A(M_{C}(2,\mc{L}))\]
 where $H^{2i}_A(Y)$ denotes the vector sub-space of \emph{algebraic} cohomology classes
 in $H^{2i}(Y)$ and 
 $H^*_A(Y)= \oplus_i H^{2i}_A(Y)$.
  Recall by \cite[Theorem $17.5.1$]{birk}, for a very general curve $C$,  
 $H^*_{\mr{Hdg}}(\mr{Jac}(C))$ is generated (as a $\mb{Q}$-vector space) by the theta divisor.
 Hence, using $\delta_C$ we have 
 \begin{equation}\label{eq:alg4}
  H^*_{\mr{Hdg}}(\mr{Jac}(C))=H^*_A(\mr{Jac}(C))
 \mbox{ and } H^*_{\mr{Hdg}}(M_{C}(2,\mc{L}))=H^*_A(M_{C}(2,\mc{L})).
 \end{equation}

 \subsection{Limit mixed Hodge structures for families of Jacobians}\label{subsecal1}
 Keep notations \S \ref{tor33}. Recall, the degenerating family 
 \[\pi_1: \mc{X} \to \Delta\]
 of curves which is smooth over the punctured disc $\Delta^*$ and the central 
 fiber is the irreducible nodal curve $X_0$. There exists a family of 
 Jacobians of curves associated to the family $\pi_1$, i.e. 
 \[\pi_4: \mbf{J}_{\mc{X}_{\Delta^*}} \to \Delta^*\]
 is a smooth, projective morphism such that for every $s \in \Delta^*$, the fiber 
 $\pi_4^{-1}(s)$ is the Jacobian $\mr{Jac}(\mc{X}_s)$ of the curve $\mc{X}_s:=\pi_1^{-1}(s)$.
By \eqref{eq:jac}, there is a natural isomorphism of pure Hodge structures
 \[\delta_s: H^1(\mr{Jac}(\mc{X}_s), \mb{Z}) \to 
 H^1(\mc{X}_s,\mb{Z}). \]
 This induces an isomorphism of local systems 
 \begin{equation}\label{eq:limjac}
  \mb{H}^1_{\mbf{J}_{\Delta^*}} \xrightarrow{\sim} \mb{H}^1_{\mc{X}_{\Delta^*}}
 \end{equation}
which induces an isomorphism of Hodge bundles from $\mc{H}^1_{\mbf{J}_{\Delta^*}}$ to 
$\mc{H}^1_{\mc{X}_{\Delta^*}}$ i.e., preserves the Hodge filtration $F^p$ (see \S \ref{sec3} for notations).
Composing the above isomorphism with $\Phi_{\Delta^*}$ from \S \ref{sec:mon}, we get an isomorphism of local 
systems 
\[\delta_{\Delta^*}: \mb{H}^1_{\mbf{J}_{\Delta^*}} \xrightarrow[\sim]{\eqref{eq:limjac}} \mb{H}^1_{\mc{X}_{\Delta^*}} \xrightarrow[\sim]{\Phi_{\Delta^*}} \mb{H}^3_{\mc{G}(2,\mc{L})_{\Delta^*}}.\]
which induces an isomorphism of Hodge bundles from $\mc{H}^1_{\mbf{J}_{\Delta^*}}$ to 
$\mc{H}^3_{\mc{G}(2,\mc{L})_{\Delta^*}}$.
Since the monodromy action on isomorphic local systems is the same, the weight filtrations 
associated to the two limit mixed Hodge structures are the same. In particular, we have 

\begin{thm}
 The morphism $\delta_{\Delta^*}$
induces an isomorphism of mixed Hodge structures:
\[\delta_\infty: H^1(\mbf{J}_\infty)(-1) \xrightarrow{\sim} H^3(\mc{G}(2,\mc{L})_\infty)\]
where $\mbf{J}_\infty$ (resp. $\mc{G}(2,\mc{L})_\infty$) is the base change of $\mbf{J}_{\Delta^*}$ (resp. 
$\mc{G}(2,\mc{L})_{\Delta^*}$) to the upper half plane $\mf{h}$ (see Definition \ref{defiunc} for notation) and 
the mixed Hodge structures on $H^1(\mbf{J}_\infty)$ and $H^3(\mc{G}(2,\mc{L})_\infty)$ are as 
defined in \S \ref{sec3}.
\end{thm}

\begin{proof}
 The proof is identical to the proof of Theorem \ref{ner01} (or \cite[Proposition $4.1$]{indpre}) 
 after replacing 
$\mc{X}$ in the statement with $\mbf{J}$.
\end{proof}

 \begin{rem}\label{rem:surj}
   Since cup-product is a morphism of mixed Hodge structures (see 
 \cite[Corollary $5.45$]{pet}), we have the following morphism of mixed Hodge structures for each $i \ge 0$:
 \[\Phi_\infty^{(i)}: H^i(\mbf{J}_\infty)(-i) = \bigwedge^i H^1(\mbf{J}_\infty)(-1) 
  \xrightarrow{\delta_\infty} \bigwedge^i H^3(\mc{G}(2,\mc{L})_\infty) \xrightarrow{\cup} H^{3i}(\mc{G}(2,\mc{L})_\infty),\]
  where $\cup$ denotes the cup-product morphism.
Choose a point $s' \in \mf{h}$ such that $e(s')=s$. Since $\mf{h}$ is contractible, 
 by the Ehresmann's lemma, there is a natural isomorphism (induced by the inclusion 
 of the fiber over $s'$ into $\mc{G}(2,\mc{L})_\infty$):
 \begin{equation}\label{eq:alg3}
  \phi_s^i: H^i(\mc{G}(2,\mc{L})_\infty,\mb{Z}) \to H^i(M_{\mc{X}_s}(2,\mc{L}_s),\mb{Z})
 \end{equation}
 commuting with the cup-product (cup-product commutes with pull-back by continuous maps)
\[ \mbox{i.e., }\, \, \phi_s^i(\xi_1) \cup \phi_s^j(\xi_2)=\phi_s^{i+j}(\xi_1 \cup \xi_2)
 \mbox{ for } \xi_1 \in H^i(\mc{G}(2,\mc{L})_\infty,\mb{Z}) \mbox{ and }
 \xi_2 \in H^j(\mc{G}(2,\mc{L})_\infty,\mb{Z}).\]
   Denote by $\alpha_\infty:=(\phi^2_s)^{-1}(\alpha_s)$ and $\beta_\infty:=(\phi^4_s)^{-1}(\beta_s)$,
   where $\alpha_s \in H^2(M_{\mc{X}_s}(2,\mc{L}_s),\mb{Z})$ and $\beta_s \in  H^4(M_{\mc{X}_s}(2,\mc{L}_s),\mb{Z})$
   arise as linear combination of the Chern classes of a universal bundle $\mc{U}_s$ for $M_{\mc{X}_s}(2,\mc{L}_s)$ 
   (see \cite[p. $337$]{new1} for an explicit description). As $\mc{U}_s$ is the restriction to the fiber over $s$
   of the universal bundle $\mc{U}$ as in \S \ref{sec:mon} for $\mc{G}(2,\mc{L})_{\Delta^*}$, 
   the classes $\alpha_s$ and $\beta_s$ are monodromy invariant. Hence, $\alpha_\infty$ and $\beta_\infty$ does 
   not depend on the choice of $s$.
      Then the morphism $\Phi_{\infty}^{(i)}$ induces a (graded) ring homomorphism  
   \[\Phi_\infty:\, \mb{Q}[\alpha_\infty, \beta_\infty] \otimes \left(\bigoplus_i H^i(\mbf{J}_\infty)(-i)\right) \xrightarrow{\mr{id} \times \oplus \Phi_\infty^{(i)}} H^*(\mc{G}(2,\mc{L})_{\infty})\]
   which induces a morphism of mixed Hodge structures on each graded piece ($\alpha_\infty$ and $\beta_\infty$ are 
   considered as Hodge classes of weight $2$ and $4$, respectively).
   Since $H^*(M_{\mc{X}_s}(2,\mc{L}_s))$ is generated by $\alpha_s, \beta_s$ and $H^3(M_{\mc{X}_s}(2,\mc{L}_s))$ (\cite[Theorem 1]{new1}) and 
   $\phi_s^i$ commutes with cup-product, the cohomology ring $H^*(\mc{G}(2,\mc{L})_\infty)$ is generated 
   by $\alpha_\infty, \beta_\infty$ and $H^3(\mc{G}(2,\mc{L})_\infty)$.
 Since the morphism $\delta_\infty$ is an isomorphism, this implies that the morphism $\Phi_\infty$ above is  surjective.
 \end{rem}

\subsection{Compactified Jacobians of curves}\label{sec:comp}
Let $X_0$ be an irreducible nodal curve. The \emph{compactified Jacobian} $\ov{J}_{X_0}$ parametrizes rank $1$, degree $0$ torsion-free sheaves on $X_0$
(see \cite[Theorem $3.3$]{lant1}). The compactified Jacobian is a semi-normal, projective variety. In particular, the normalization map 
\[h: \widetilde{J}_{X_0} \to \ov{J}_{X_0}\]
is the desingularization of $\ov{J}_{X_0}$ i.e., $\widetilde{J}_{X_0}$ is non-singular and the pre-image under $h$ of the singular locus of $\ov{J}_{X_0}$ 
consists of two disjoint non-singular divisors, say $D_1$ and $D_2$ in $\widetilde{J}_{X_0}$
which are both isomorphic to the Jacobian $\mr{Jac}(\widetilde{X}_0)$, of the normalization $\widetilde{X}_0$ of the curve $X_0$.
Moreover, $\widetilde{J}_{X_0}$ is a $\mb{P}^1$-bundle over $\mr{Jac}(\widetilde{X}_0)$. See \cite[p. $567$]{narafac} for a nice description.

One can also consider the compactified Jacobian $\ov{J}_{X_0}$ as a degeneration of Jacobians of smooth curves. 
In particular, given a family $\pi_1$ of smooth curves degenerating to $X_0$ as above, there exists 
a flat, projective family (see \cite[Theorem $3.3$]{lant1}) \[\pi_5: \ov{\mbf{J}} \to \Delta\]
such that for any $s \in \Delta^*$, $\pi_5^{-1}(s) \cong \mr{Jac}(\mc{X}_s)$ and the central fiber $\pi_5^{-1}(0)=\ov{J}_{X_0}$.
Blowing up the family $\ov{\mbf{J}}$ along the singular locus of the central fiber $\ov{J}_{X_0}$ followed by an etale base change, we get 
a flat projective family (this is the \emph{semi-stable reduction} of $\pi_5$, see \cite[Example $(2.15)$]{ste2}) \[\widetilde{\pi}_5: \ov{\mbf{J}}' \to \Delta\]
such that the fibers over $\Delta^*$ coincides with that of $\pi_5$ above and the central fiber $\widetilde{\pi}_5^{-1}(0)$ is a reduced 
simple normal crossings divisor consisting of two irreducible components $\widetilde{J}_{X_0}$ and $E$,
where $E$ is a $\mb{P}^1$-bundle over $\mr{Jac}(\widetilde{X}_0)$ and $E \cap \widetilde{J}_{X_0}= D_1 \cup D_2$.

\subsection{Limit Hodge conjecture}
 Denote by $H^{2p}_{\mr{Hdg}}(\mbf{J}_\infty):= H^{p,p} \mr{Gr}^W_{2p} H^{2p}(\mbf{J}_{\infty})$. 
 Using the description of compactified Jacobians in families given above we prove:
 
 \begin{thm}\label{thm:lim}
  For a very general nodal curve $X_0$ as defined in the Introduction, 
  we have $H^{2p}_{\mr{Hdg}}(\mbf{J}_\infty) \cong \mb{Q}$ for each $p$.
 \end{thm}

 \begin{proof}
  For the proof of this statement we use the exact sequence \eqref{ntor02}. Notations as in \S \ref{sec:comp}.
  Applying Corollary \ref{ntor03} to $\widetilde{\pi}_5$ above, we get the exact sequence of pure Hodge structures:
  \begin{equation}\label{eq:spe}
   H^{2p-2}(D_1 \cup D_2)(-1) \xrightarrow{f_{2p}} \mr{Gr}^W_{2p} H^{2p}(\widetilde{J}_{X_0} \cup E) \xrightarrow{\mr{sp}_{2p}} \mr{Gr}^W_{2p} H^{2p}(\mbf{J}_\infty) \to 0.
  \end{equation}
  To compute the middle term in this exact sequence we use the Mayer-Vietoris sequence. In particular, we have an exact sequence of pure Hodge structures:
  \begin{equation}\label{eq:may}
   0 \to \mr{Gr}^W_{2p}  H^{2p}(\widetilde{J}_{X_0} \cup E) \to  H^{2p}(\widetilde{J}_{X_0}) \oplus H^{2p}(E) \to H^{2p}(D_1 \cup D_2).
  \end{equation}
  Recall, $E$ and $\widetilde{J}_{X_0}$ are $\mb{P}^1$-bundles over $\mr{Jac}(\widetilde{X}_0)$.
  Denote by 
  \[q_{_E}: E \to \mr{Jac}(\widetilde{X}_0) \mbox{ and } q: \widetilde{J}_{X_0} \to \mr{Jac}(\widetilde{X}_0)\]
  the natural bundle morphisms.  Let $\xi_0 \in H^2(\mb{P}^1)$ be a generator.
  By the Leray-Hirsch theorem (see \cite[Theorem $7.33$]{v4}), we have for any closed point 
  $y \in \mr{Jac}(\widetilde{X}_0)$ and $p \ge 0$:
  \begin{align}
   & H^{2p}(E) \cong \bigoplus_{j \ge 0} (H^j(E_y) \otimes q_{_E}^*H^{2p-j}(\mr{Jac}(\widetilde{X}_0))) \cong 
   q_{_E}^* H^{2p}(\mr{Jac}(\widetilde{X}_0)) \oplus q_{_E}^* H^{2p-2}(\mr{Jac}(\widetilde{X}_0))\xi_0
  \label{eq:ler01}\\
  & H^{2p}(\widetilde{J}_{X_0}) \cong \bigoplus_{j \ge 0} (H^j(\widetilde{J}_{X_0, y}) \otimes q^*H^{2p-j}(\mr{Jac}(\widetilde{X}_0)))
  \cong q^* H^{2p}(\mr{Jac}(\widetilde{X}_0)) \oplus q^* H^{2p-2}(\mr{Jac}(\widetilde{X}_0))\xi_0 \label{eq:ler02}
  \end{align}
   Denote by $q_i: D_i \hookrightarrow \widetilde{J}_{X_0}$ and $q_{_{E,i}}: D_i \hookrightarrow E$ the closed immersions.  
  As $q \circ q_i$ and $q_{_E} \circ q_{_{E,i}}$ are isomorphisms, $q_i^*$ (resp. $q_{_{E,i}}^*$) maps 
  $q^* H^{2p}(\mr{Jac}(\widetilde{X}_0))$ (resp. $q_{_E}^* H^{2p}(\mr{Jac}(\widetilde{X}_0))$) isomorphically to 
   $H^{2p}(D_i)$. 
Using the decompositions \eqref{eq:ler01} and \eqref{eq:ler02}, we then conclude that the last (restriction)
  morphism of \eqref{eq:may} is surjective.
   Since (by assumption) $X_0$ is very general, we have $H^{2p}_{\mr{Hdg}}(\mr{Jac}(\widetilde{X}_0)) \cong \mb{Q}$
  for every $p \ge 0$. Hence, using the decompositions \eqref{eq:ler01} and \eqref{eq:ler02} we have 
  \[\dim H^{2p-2}_{\mr{Hdg}}(D_1 \cup D_2)=\dim H^{2p}_{\mr{Hdg}}(D_1 \cup D_2) = 2 = \dim H^{2p}_{\mr{Hdg}}(\widetilde{J}_{X_0})=\dim H^{2p}_{\mr{Hdg}}(E).\]
  Then the surjectivity on the right of \eqref{eq:may} implies that 
  $\dim H^{p,p}\mr{Gr}^W_{2p} H^{2p}(\widetilde{J}_{X_0} \cup E)=2$. 
 Using the exact sequence \eqref{eq:spe} it therefore suffices to check that the kernel of $f_{2p}$ when restricted to 
 $H^{2p-2}_{\mr{Hdg}}(D_1 \cup D_2)(-1)$ is one dimensional.
  Using the projection formula (see \cite[Lemma B.$26$]{pet}) and the identification described in the 
  Leray-Hirsch theorem (identifying certain cohomology classes with their restriction to $y$) 
  we have for any $\gamma \in H^{2p-2}(\mr{Jac}(\widetilde{X}_0))$, that 
  \[q_{i,_*} \circ q_i^* \circ q^*(\gamma)=q^*(\gamma)\xi_0 \mbox{ and } 
  q_{_{E,i},_*} \circ q_{_{E,i}}^* \circ q_{_E}^*(\gamma)=q_{_E}^*(\gamma)\xi_0.\] 
  Taking $\gamma= \theta^{p-1}$, where $\theta \in H^2(\mr{Jac}(\widetilde{X}_0))$ is the 
  class of the theta divisor, we then observe that the kernel of 
  \[f_{2p}: H^{2p-2}_{\mr{Hdg}} (D_1 \cup D_2)(-1) \to H^{p,p} \mr{Gr}^W_{2p}H^{2p}(\widetilde{J}_{X_0} \cup E)\]
  is one dimensional, generated by (a scalar multiple of) the pull-back of $\theta^{p-1}$
  (recall from \S \ref{sec3}
 that $f_{2p}$ is  induced by the Gysin morphism).
 This proves the theorem.
  \end{proof}

 \begin{note}\label{note:hdg}
  Let $\phi_s^i$ as in \eqref{eq:alg3}.
  Denote by $H^{2p}_{\mr{Hdg}}(\mc{G}(2,\mc{L})_\infty):= H^{p,p} \mr{Gr}^W_{2p} H^{2p}(\mc{G}(2,\mc{L})_{\infty})$
  and  \begin{align*}
H^{2p}_A(\mc{G}(2,\mc{L})_\infty) &= \left\{\begin{array}{l}
       \sigma \in H^{2p}(\mc{G}(2,\mc{L})_\infty) \mbox{ such that for a very general } s \in \Delta^* \mbox{ and the }\\
 \mbox{natural isomorphism } \phi_s^{2p}: H^{2p}(\mc{G}(2,\mc{L})_\infty) \to H^{2p}(\mc{G}(2,\mc{L})_s),\\
 \phi_s^{2p}(\sigma) \in H^{2p}_A(\mc{G}(2,\mc{L})_s).    \end{array}\right\}
\end{align*}
 \end{note}

 \begin{cor}\label{cor:lim}
 Let $X_0$ be as in Theorem \ref{thm:lim}.
   Let $\theta_\infty$ be a generator of $H^{2}_{\mr{Hdg}}(\mbf{J}_\infty)$ (see Theorem \ref{thm:lim}). 
      Then, the graded ring $H^*_{\mr{Hdg}}(\mc{G}(2,\mc{L})_\infty):= \oplus_p H^{2p}_{\mr{Hdg}}(\mc{G}(2,\mc{L})_\infty)$ is generated (as a ring over $\mb{Q}$) by $\alpha_\infty, \beta_\infty$ and (the image of) $\theta_\infty$.
      In particular, $H^{2p}_A(\mc{G}(2,\mc{L})_\infty)=H^{2p}_{\mr{Hdg}}(\mc{G}(2,\mc{L})_\infty)$.
 \end{cor}

 The above corollary is what we call the ``limit'' Hodge conjecture in the family of Gieseker moduli spaces. 
 Also note that $H^*_{\mr{Hdg}}(\mc{G}(2,\mc{L})_\infty)$ is not a polynomial ring in $3$ variables i.e.,
 there are relations between $\alpha_\infty, \beta_\infty$ and $\theta_\infty$ (see \cite[p. $408$]{kingn}).

 \begin{proof}[Proof of Corollary \ref{cor:lim}]
 Recall from Remark \ref{rem:surj} that $\Phi_\infty^{(i)}$ is a morphism of mixed Hodge structures and $\Phi_\infty$,
 which is the direct sum over all the graded pieces, is surjective. This implies that we have a surjective morphism 
 induced by $\Phi_\infty$:
 \[\Phi_\infty: \mb{Q}[\alpha_\infty, \beta_\infty] \otimes_{\mb{Q}} H^*_{\mr{Hdg}}(\mbf{J}_\infty) \twoheadrightarrow 
  H^*_{\mr{Hdg}}(\mc{G}(2,\mc{L})_\infty),\, \mbox{ where } H^*_{\mr{Hdg}}(\mbf{J}_\infty):=
  \bigoplus_{p \ge 0} H^{2p}_{\mr{Hdg}}(\mbf{J}_\infty). \]
By Theorem \ref{thm:lim}, $H^*_{\mr{Hdg}}(\mbf{J}_\infty)$ is generated by $\theta_\infty$. 
Hence, $H^*_{\mr{Hdg}}(\mc{G}(2,\mc{L})_\infty)$ is generated  by $\alpha_\infty, \beta_\infty$ and (the image of) $\theta_\infty$. 
This proves the first part of the corollary.

Observe that for a very general $s \in \Delta^*$, 
$H^*_A(\mc{G}(2,\mc{L})_s)=H^*_A(M_{\mc{X}_s}(2,\mc{L}_s))$ is generated as a ring over $\mb{Q}$, 
by $\alpha_s, \beta_s$ and (the image of) a generator $\theta_s$ of $H^2_A(\mr{Jac}(\mc{X}_s))$
(see \S \ref{recsmcase}). Since $\alpha_\infty, \beta_\infty$ and (the image of) $\theta_\infty$
maps to $\alpha_s, \beta_s$ and (the image of) $\theta_s$, respectively (upto multiplication by a scalar) by the morphism 
$\phi_s^i$ as in \eqref{eq:alg3} for $i=2, 4, 6$, we conclude that $H^*_A(\mc{G}(2,\mc{L})_\infty)$
is generated by $\alpha_\infty, \beta_\infty$ and (the image of) $\theta_\infty$. 
Using the first part, we conclude 
$H^{*}_A(\mc{G}(2,\mc{L})_\infty)=H^{*}_{\mr{Hdg}}(\mc{G}(2,\mc{L})_\infty)$ as graded rings over $\mb{Q}$.
This proves the corollary.
 \end{proof}

\section{Main Results}\label{sec5}
 
  Assume that $X_0$ is a very 
  general, irreducible nodal curve  as defined in the Introduction. Note that, by the genericity of $X_0$,
  we can choose the family of curves $\pi_1$ as in 
  \S \ref{tor33} such that for a very general $s \in \Delta^*$, 
  we have $H^*_{\mr{Hdg}}(\mr{Jac}(\mc{X}_s))$ is generated by the cohomology class of the theta 
  divisor $\theta_s \in H^2_{\mr{Hdg}}(M_{\mc{X}_s}(2,\mc{L}_s))$.
  We fix such a family $\pi_1$, for this section.
    Let $j_0:\mc{G}_0 \hookrightarrow \mc{G}_{X_0}(2,\mc{L}_0)$ and $j_1:\mc{G}_1 \hookrightarrow \mc{G}_{X_0}(2,\mc{L}_0)$ be
  the natural inclusions. Define the \emph{algebraic cohomology groups} on $\mc{G}_{X_0}(2,\mc{L}_0)$ as:
  \[
  H^{2i}_A(\mc{G}_{X_0}(2,\mc{L}_0)) = \left\{\begin{array}{l}
    \sigma \in \mr{Gr}^W_{2i} H^{2i}(\mc{G}_{X_0}(2,\mc{L}_0)) \mbox{ such that } j_t^*(\sigma) \in 
    H^{2i}_A(\mc{G}_t),\, \, t=0, 1          \end{array}\right\}
    \]

 We first obtain an  exact sequence analogous to \eqref{ntor09} for algebraic classes.

 \begin{lem}\label{lem:alg}
  For any $i \ge 0$, the algebraic cohomology group $H^{2i}_A(\mc{G}_{X_0}(2,\mc{L}_0))$ sits in the following exact sequence:
   \begin{equation}\label{eq:alg1}
   0 \to H^{2i-4}_A(M_{\widetilde{X}_0}(2,\widetilde{\mc{L}}_0)) \xrightarrow{h_{2i}}
   H^{2i-2}_A(\mc{G}_0 \cap \mc{G}_1) \xrightarrow{f_{2i}} H^{2i}_A(\mc{G}_{X_0}(2,\mc{L}_0)) \xrightarrow{\mr{sp}_{2i}}
   H^{2i}_A(\mc{G}(2,\mc{L})_\infty)  \to  0.
   \end{equation} 
 \end{lem}

 \begin{proof}
 The exact sequence \eqref{eq:alg1} will be induced by the exact sequence \eqref{ntor09}.
  We first claim that the restriction of the 
  specialization morphism $\mr{sp}_{2i}$ to the algebraic cohomology group
  $H^{2i}_A(\mc{G}_{X_0}(2,\mc{L}_0))$
 factors through $H^{2i}_A(\mc{G}(2,\mc{L})_\infty)$.
  Let $\gamma \in H^{2i}_A(\mc{G}_{X_0}(2,\mc{L}_0))$. 
  By definition,  for $t=0,1$ we have $j_t^*(\gamma) \in H^{2i}_A(\mc{G}_t)$. 
  In particular, $j_t^*(\gamma) \in H^{i,i}(\mc{G}_t)$ for $t=0,1$.
  By the Mayer-Vietoris sequence \eqref{eq:viet}, this implies 
   $\gamma \in H^{i,i}\mr{Gr}^W_{2i}H^{2i}(\mc{G}_{X_0}(2,\mc{L}_0), \mb{C})$.
 As $\mr{sp}_{2i}$ is a morphism of mixed Hodge structures, this implies  
 $\mr{sp}_{2i}(\gamma) \in H^{2i}_{\mr{Hdg}}(\mc{G}(2, \mc{L})_\infty)$ (see Notation \ref{note:hdg}).
 By Corollary \ref{cor:lim}, $H^{2i}_{A}(\mc{G}(2, \mc{L})_\infty)=H^{2i}_{\mr{Hdg}}(\mc{G}(2, \mc{L})_\infty)$.
 Hence, $\mr{sp}_{2i}(\gamma) \in H^{2i}_A(\mc{G}(2,\mc{L})_\infty)$. This proves our claim.    
 
 Next we see that restriction of the specialization morphism to algebraic classes  is surjective. 
 Let $\alpha_s \in H^2(M_{\mc{X}_s}(2,\mc{L}_s))$ and $\beta_s \in H^4(M_{\mc{X}_s}(2,\mc{L}_s))$ be as in 
 Remark \ref{rem:surj} and $\theta_s \in H^6(M_{\mc{X}_s}(2,\mc{L}_s))$ be as in the proof of
 Corollary \ref{cor:lim} (same as $\phi_s^6(\theta_\infty)$). Note that, $\alpha_s, \beta_s$ and $\theta_s$ arise as 
 restriction to the fiber over $s$, of linear combinations of the Chern classes of the universal bundle $\mc{U}$
 on $\mc{G}(2,\mc{L})_{\Delta^*}$ (see Remark \ref{rem:surj}). 
 Hence, $\alpha_s, \beta_s$ and $\theta_s$ are monodromy invariant (they are restrictions of global sections of 
  local systems).
 By taking closure 
 (over the unit disc) of the Chern classes of $\mc{U}$, the classes $\alpha_s, \beta_s$ and $\theta_s$ extend as algebraic 
 classes to the central fiber $\mc{G}_{X_0}(2,\mc{L}_0)$. This implies that there exists 
 $\alpha_0 \in H^2_A(\mc{G}_{X_0}(2,\mc{L}_0)), \beta_0 \in  H^4_A(\mc{G}_{X_0}(2,\mc{L}_0))$
 and $\theta_0 \in  H^6_A(\mc{G}_{X_0}(2,\mc{L}_0))$ such that the specialization morphism 
 maps $\alpha_0$ (resp. $\beta_0, \theta_0$) to $\alpha_\infty$ (resp. $\beta_\infty, \theta_\infty$) in 
 $H^*_A(\mc{G}(2,\mc{L})_\infty)$. By Corollary \ref{cor:lim}, $H^*_A(\mc{G}(2,\mc{L})_\infty)$
 is generated by $\alpha_\infty, \beta_\infty$ and $\theta_\infty$. Hence, 
 \[\mr{sp}_{2i}: H^{2i}_A(\mc{G}_{X_0}(2,\mc{L}_0)) \to H^{2i}_A(\mc{G}(2,\mc{L})_\infty)\]
 is surjective. By the exactness of \eqref{ntor09}, $\ker(\mr{sp}_{2i}) \cap H^{2i}_A(\mc{G}_{X_0}(2,\mc{L}_0))$
 consists of classes \[\sigma \in H^{2i}_{\mr{Hdg}}(\mc{G}_{X_0}(2,\mc{L}_0))\, \mbox{ such that }\, 
 j_t^*\sigma \in 
 H^{2i}_A(\mc{G}_t)\, \mbox{ and }\]
 there exists $\gamma \in H^{2i-2}(\mc{G}_0 \cap \mc{G}_1)$ such that 
 $i_{t,_*}(\gamma)=j_t^*(\sigma)$, where $i_t: \mc{G}_0 \cap \mc{G}_1 \hookrightarrow \mc{G}_t$ for $t=0,1$ is 
 the natural inclusion. Since $i_{t,_*}$ is a Gysin morphism (of pure Hodge structures) from 
 $H^{2i-2}(\mc{G}_0 \cap \mc{G}_1)(-1)$ to $H^{2i}(\mc{G}_t)$, $\gamma$ is a Hodge class (as $j_t^*\sigma$
 is a Hodge class). Moreover, by assumption $M_{\widetilde{X}_0}(2,\widetilde{\mc{L}}_0)$ satisfies the 
 Hodge conjecture and so does $\mc{G}_0 \cap \mc{G}_1$ as it is a $\mb{P}^1 \times \mb{P}^1$-bundle over $M_{\widetilde{X}_0}(2,\widetilde{\mc{L}}_0)$.  In particular, 
 $\gamma$ is an algebraic class i.e., $\gamma \in H^{2i-2}_A(\mc{G}_0 \cap \mc{G}_1)$. 
 
 Similarly, as $h_{2i}$ from $H^{2i-4}(M_{\widetilde{X}_0}(2,\widetilde{\mc{L}}_0))(-2)$ to 
 $H^{2i-2}(\mc{G}_0 \cap \mc{G}_1)(-1)$ is a morphism of pure Hodge structures, the exactness of 
 \eqref{ntor09} implies that 
 $\ker(f_{2i}) \cap 
 H^{2i-2}_A(\mc{G}_0 \cap \mc{G}_1)$ equals the image of the morphism:
 \[H^{2i-4}_A(M_{\widetilde{X}_0}(2,\widetilde{\mc{L}}_0))= H^{2i-4}_{\mr{Hdg}}(M_{\widetilde{X}_0}(2,\widetilde{\mc{L}}_0))(-2) \xrightarrow{h_{2i}} H^{2i-2}_{\mr{Hdg}}(\mc{G}_0 \cap \mc{G}_1)(-1)=H^{2i-2}_A(\mc{G}_0 \cap \mc{G}_1).\]
 The injectivity of $h_{2i}$ follows from the exactness of \eqref{ntor09}. This proves the lemma. 
 \end{proof}


 \begin{thm}\label{hod02}
 For the central fibre $\mc{G}_{X_0}(2,\mc{L}_0)$ of the relative Gieseker's moduli space  we have,
 \[H^*_{\mr{Hdg}}(\mc{G}_{X_0}(2,\mc{L}_0))=H^*_A(\mc{G}_{X_0}(2,\mc{L}_0)).\] 
 Moreover, the restriction morphism $j_0^*:H^i(\mc{G}_{X_0}(2,\mc{L}_0))  \to H^i(\mc{G}_0)$ is surjective and the irreducible components $\mc{G}_{0}$ and $\mc{G}_1$ satisfy the Hodge conjecture. 
 
 In particular, the Hodge conjecture holds for a desingularization of $U_{X_0}(2,\mc{L}_0)$.
 \end{thm}

  \begin{proof}
 Since $\mr{Gr}^W_{2i+1}H^{2i}(\mc{G}(2,\mc{L})_\infty)$ is pure of weight $2i+1$ and 
 the morphisms in \eqref{ntor09} are morphisms of mixed Hodge structure, we have the following exact sequence:
\begin{equation}\label{eq:alg2}
  0 \to H^{2i-4}_{\mr{Hdg}}(M_{\widetilde{X}_0}(2,\widetilde{\mc{L}}_0)) \to H^{2i-2}_{\mr{Hdg}}(\mc{G}_0 \cap \mc{G}_1) \xrightarrow{f_{2i}} H^{2i}_{\mr{Hdg}}(\mc{G}_{X_0}(2,\mc{L}_0)) \xrightarrow{\mr{sp}_{2i}} H^{2i}_{\mr{Hdg}}(\mc{G}(2,\mc{L})_\infty)  \to 0.
\end{equation}
We observed in the proof of Lemma \ref{lem:alg} that using the Mayer-Vietoris sequence \eqref{eq:viet},
there is a natural inclusion from $H^{2i}_A(\mc{G}_{X_0}(2,\mc{L}_0))$ to 
$H^{2i}_{\mr{Hdg}}(\mc{G}_{X_0}(2,\mc{L}_0))$.
By Corollary \ref{cor:lim}, we have $H^{2i}_A(\mc{G}(2,\mc{L})_\infty)=H^{2i}_{\mr{Hdg}}(\mc{G}(2,\mc{L})_\infty)$.
Then combining \eqref{eq:alg2} with \eqref{eq:alg1} we have the following diagram of short exact sequences:
  {\small{\[\begin{diagram}\label{d1}
 0 &\rTo& \mr{ker}(\mr{sp}_{2i}) \cap H^{2i}_A(\mc{G}_{X_0}(2,\mc{L}_0)) & \rTo^{f_{2i}} &
 H^{2i}_A(\mc{G}_{X_0}(2,\mc{L}_0)) &\rTo^{\mr{sp}_{2i}}& H^{2i}_A(\mc{G}(2,\mc{L})_\infty)  &\rTo& 0 \\
 & & \dInto & \circlearrowright &\dInto &\circlearrowright &\dTo^{\mr{id}}  & & \\
    0 &\rTo& \mr{ker}(\mr{sp}_{2i}) \cap H^{2i}_{\mr{Hdg}}(\mc{G}_{X_0}(2,\mc{L}_0)) &\rTo^{f_{2i}} & H^{2i}_{\mr{Hdg}}(\mc{G}_{X_0}(2,\mc{L}_0)) &\rTo^{\mr{sp}_{2i}}&
    H^{2i}_{\mr{Hdg}}(\mc{G}(2,\mc{L})_\infty)  &\rTo& 0\\
   \end{diagram}\]}}
 To show that the middle arrow in the above diagram is an isomorphism, it suffices to prove that the first  
 vertical arrow is an isomorphism. 
 This follows from comparing exact sequences \eqref{eq:alg1} and \eqref{eq:alg2} and using that Hodge classes are algebraic both for $M_{\widetilde{X}_0}(2,\widetilde{\mc{L}}_0)$ by the discussion in \S \ref{recsmcase} and for 
 $\mc{G}_0\cap \mc{G}_1$ which is a $\mb{P}^1\times \mb{P}^1$-bundle over $M_{\widetilde{X}_0}(2,\widetilde{\mc{L}}_0)$.
 Hence, $H^*_{\mr{Hdg}}(\mc{G}_{X_0}(2,\mc{L}_0))=H^*_A(\mc{G}_{X_0}(2,\mc{L}_0))$. This completes the proof of the first part of the theorem.

We now show that the restriction morphism $j_0^*:H^i(\mc{G}_{X_0}(2,\mc{L}_0))  \to H^i(\mc{G}_0)$ is surjective. Consider the following exact sequences of pure Hodge structures (\cite[Proposition $5.46$]{pet}):
\begin{eqnarray}\label{gen02}
& 0 \to \mr{Gr}^W_iH^i(\mc{G}_{X_0}(2,\mc{L}_0),\mc{G}_0) \to \mr{Gr}^W_iH^i(\mc{G}_{X_0}(2,\mc{L}_0)) \to H^i(\mc{G}_0) \to \\ \nonumber
 & \to  \mr{Gr}^W_{i}H^{i+1}(\mc{G}_{X_0}(2,\mc{L}_0),\mc{G}_0) \to \mr{Gr}^W_{i}H^{i+1}(\mc{G}_{X_0}(2,\mc{L}_0)) \to 0, 
\end{eqnarray}
\begin{equation}\label{gen03}
0 \to \mr{Gr}^W_iH^i(\mc{G}_1,\mc{G}_0 \cap \mc{G}_1) \to H^i(\mc{G}_1) \xrightarrow{i_1^*} H^i(\mc{G}_0 \cap \mc{G}_1) \to 
 \mr{Gr}^W_{i}H^{i+1}(\mc{G}_1,\mc{G}_0 \cap \mc{G}_1) \to 0.
 \end{equation}
  Note that $r_1$ (resp. $r_3$) is injective (resp. surjective) with cokernel (resp. kernel) 
 isomorphic to $\mb{Q}(\xi_1-\xi_2)$ (resp. $\mb{Q}(\xi'')^3$), where notations as in \S \ref{note:ner01}. 
 Using \eqref{ntor04} and \eqref{ntor05}, this implies that the restriction morphism 
 \[\ker(i_1^*) \cong H^{i-6}(M_{\widetilde{X}_0}(2,\widetilde{\mc{L}}_0))(\xi'')^3 \mbox{ and } \mr{coker}(i_1^*) \cong H^{i-2}(M_{\widetilde{X}_0}(2,\widetilde{\mc{L}}_0))(\xi_1-\xi_2).\]
 The exact sequence \eqref{gen03} then implies, 
 \[\mr{Gr}^W_iH^i(\mc{G}_1, \mc{G}_0 \cap \mc{G}_1) \cong H^{i-6}(M_{\widetilde{X}_0}(2,\widetilde{\mc{L}}_0))(\xi'')^3 \mbox{ and }
 \mr{Gr}^W_{i}H^{i+1}(\mc{G}_1, \mc{G}_0 \cap \mc{G}_1) \cong H^{i-2}(M_{\widetilde{X}_0}(2,\widetilde{\mc{L}}_0))(\xi_1-\xi_2).\]
 By \cite[Example B.$5(2)$]{pet}, the pair $\mc{G}_0$ and $\mc{G}_1$ form an
 excisive couple i.e., the induced excision map $j_1^*: H^i(\mc{G}_{X_0}(2,\mc{L}_0),\mc{G}_0) \to H^i(\mc{G}_1,\mc{G}_0 \cap \mc{G}_1)$ is an isomorphism of mixed Hodge structures. 
  Then, the exact sequence \eqref{gen02} becomes
  \begin{eqnarray}
   &0 \to H^{i-6}(M_{\widetilde{X}_0}(2,\widetilde{\mc{L}}_0))(\xi'')^3 \to \mr{Gr}^W_iH^i(\mc{G}_{X_0}(2,\mc{L}_0)) \to H^i(\mc{G}_0) \to \nonumber \\
  &\to H^{i-2}(M_{\widetilde{X}_0}(2,\widetilde{\mc{L}}_0))(\xi_1-\xi_2) \to \mr{Gr}^W_{i}H^{i+1}(\mc{G}_{X_0}(2,\mc{L}_0)) \to 0. \label{eq:alg01}
  \end{eqnarray}
 Let $N:=\dim \mc{G}_{X_0}(2,\mc{L}_0)$. Using \cite[Example $3.5$]{ste1} and Poincar\'{e} duality, we have that 
 \[\mr{Gr}^W_{i}H^{i+1}(\mc{G}_{X_0}(2,\mc{L}_0)) \cong \mr{coker}((i_0^*-i_1^*):H^i(\mc{G}_0) \oplus H^i(\mc{G}_1) \to H^i(\mc{G}_0 \cap \mc{G}_1)) \cong \]
 \[\cong \mr{ker}((i_{0,_*},i_{1,_*}):H^{2N-i-2}(\mc{G}_0 \cap \mc{G}_1) \to H^{2N-i}(\mc{G}_0) \oplus H^{2N-i}(\mc{G}_1))^\vee.\]
 Note that, $\dim M_{\widetilde{X}_0}(2,\widetilde{\mc{L}}_0))=N-3$. By Propositions \ref{ner15} combined with Poincar\'{e} duality, we then conclude that 
 \[\mr{Gr}^W_{i}H^{i+1}(\mc{G}_{X_0}(2,\mc{L}_0)) \cong 
 (H^{2N-i-4}(M_{\widetilde{X}_0}(2,\widetilde{\mc{L}}_0))(\xi_1 \oplus -\xi_2))^\vee \cong H^{i-2}(M_{\widetilde{X}_0}(2,\widetilde{\mc{L}}_0)).\]
 Therefore, the exact sequence \eqref{eq:alg01} becomes the following short exact sequence 
 (since a surjective morphism of vector spaces of same dimension is isomorphic):
 \begin{equation}\label{eq:ner02}
  0 \to  H^{i-6}(M_{\widetilde{X}_0}(2,\widetilde{\mc{L}}_0))(\xi'')^3 \to \mr{Gr}^W_iH^i(\mc{G}_{X_0}(2,\mc{L}_0)) \to H^i(\mc{G}_0) \to 0.
 \end{equation}
 The above sequence shows the surjectivity of $j_0^*$. Using this and the equality 
 $H_A^*(\mc{G}_{X_0}(2,\mc{L}_0))=H_{\mr{Hdg}}^*(\mc{G}_{X_0}(2,\mc{L}_0))$, 
the Hodge conjecture for $\mc{G}_0$ follows. Since $\mc{G}_1$ is $\p3$-bundle over $M_{\widetilde{X_0}}(2,\widetilde{\mc{L}}_0)$, it satisfies the Hodge conjecture.
     This proves the theorem.
     \end{proof}

 As an easy consequence we can compute the algebraic Poincar\'{e} polynomial of $\mc{G}_0$.

\begin{thm}\label{cor11}
The \emph{algebraic Poincar\'{e} polynomial} $P_A(\mc{G}_0)$, for $\mc{G}_0$ is given by 
 \[P_A(\mc{G}_0):=\sum H^i_A(\mc{G}_0) t^i=\frac{(1-t^g)(1-t^{g+1})(t^2(1-t^{g-1})(1+t^2)+(1-t^{g+2}))}{(1-t)(1-t^2)(1-t^3)}.\]
\end{thm}

\begin{proof}
Let $i_1$ be as in Notation \ref{note:inc}.
Using the decompositions \eqref{ntor04} and \eqref{ntor05}, we observe that 
  \[(i_1^*)^{-1}(H^{2i}_A(\mc{G}_0 \cap \mc{G}_1)) \cong H^{2i}_A(\mc{G}_1).\]
   By Theorem \ref{hod02}, the restriction morphism $j_0^*:H^i(\mc{G}_{X_0}(2,\mc{L}_0))  \to H^i(\mc{G}_0)$ is surjective. As $j_0^*$ is a morphism of mixed Hodge structures, it maps 
   $H^{2i}_A(\mc{G}_{X_0}(2,\mc{L}_0))=H^{2i}_{\mr{Hdg}}(\mc{G}_{X_0}(2,\mc{L}_0))$ surjectively to 
   $H^{2i}_{\mr{Hdg}}(\mc{G}_0)=H^{2i}_A(\mc{G}_0)$, where the equalities follow from Theorem \ref{hod02}.
   Moreover, using the exactness of \eqref{eq:viet} and  $H^{2i}_A(\mc{G}_{X_0}(2,\mc{L}_0))=H^{2i}_{\mr{Hdg}}(\mc{G}_{X_0}(2,\mc{L}_0))$, $V:=\ker(j_0^*) \cap H^{2i}_A(\mc{G}_{X_0}(2,\mc{L}_0))$
   consists of pairs $(0,\gamma) \in H^{2i}_{\mr{Hdg}}(\mc{G}_0) \oplus H^{2i}_{\mr{Hdg}}(\mc{G}_1)$
   such that $i_1^*(\gamma)=0$ i.e., $V=\ker i_1^* \cap H^{2i}_A(\mc{G}_1)$ (as $H^{2i}_A(\mc{G}_1)=
   H^{2i}_{\mr{Hdg}}(\mc{G}_1)$). Using the decompositions 
  \eqref{ntor04} and \eqref{ntor05}, we observe that $V$
  is isomorphic to $H^{2i-6}_A(M_{\widetilde{X}_0}(2,\widetilde{\mc{L}}_0))(\xi^{''})^3$ i.e.,
  we have the following short exact sequence:
  \begin{equation}\label{eq:poly}
  0 \to H^{2i-6}_A(M_{\widetilde{X}_0}(2,\widetilde{\mc{L}}_0))(\xi^{''})^3 \to H^{2i}_A(\mc{G}_{X_0}(2,\mc{L}_0)) \xrightarrow{j_0^*} H^{2i}_A(\mc{G}_0) \to 0.
  \end{equation}
   Let $H(g,t):=(1-t^g)(1-t^{g+1})(1-t^{g+2})/((1-t)(1-t^2)(1-t^3))$. By \cite[$(5.1)$]{bala2},
 \[P_A(M_{\widetilde{X}_0}(2,\widetilde{\mc{L}}_0))=H(g-1,t) \mbox{ and } P_A(M_{\mc{X}_s}(2,\mc{L}_s))=H(g,t) \mbox{ for } s \in \Delta^* \mbox{ very general}.\]
 By definition, $P_A(\mc{G}(2,\mc{L})_\infty)=P_A(M_{\mc{X}_s}(2,\mc{L}_s))$ for very general $s \in \Delta^*$.
 Using the identification \eqref{ntor04}, we have $P_A(\mc{G}_0 \cap \mc{G}_1)=H(g-1,t)(1+2t^2+t^4)$.
 The exact sequence \eqref{eq:alg1} then implies 
 \begin{equation}
  P_A(\mc{G}_{X_0}(2,\mc{L}_0))=P_A(\mc{G}_0 \cap \mc{G}_1)t^2+P_A(\mc{G}(2,\mc{L})_\infty)-P_A(M_{\widetilde{X}_0}(2,\widetilde{\mc{L}}_0))t^4.
 \end{equation}
 which equals $H(g-1,t)(t^2+t^4+t^6)+H(g,t)$.
 Finally, using the short exact sequence \eqref{eq:poly}, we conclude that
     \[P_A(\mc{G}_0)=P_A(\mc{G}_{X_0}(2,\mc{L}_0))-P_A(M_{\widetilde{X}_0}(2,\widetilde{\mc{L}}_0))t^6=H(g-1,t)(t^2+t^4)+H(g,t).\]
  Substituting for $H(g,t)$ one immediately gets the corollary.
 \end{proof}


\begin{thebibliography}{10}

\bibitem{abe}
T.~Abe.
\newblock The moduli stack of rank-two {G}ieseker bundles with fixed
  determinant on a nodal curve {II}.
\newblock {\em International Journal of Mathematics}, 20(07):859--882, 2009.

\bibitem{arb}
E.~Arbarello, M.~Cornalba, and P.~Griffiths.
\newblock {\em Geometry of algebraic curves: volume II with a contribution by
  Joseph Daniel Harris}, volume 268.
\newblock Springer Science \& Business Media, 2011.

\bibitem{AH}
M.~F. Atiyah and F.~Hirzebruch.
\newblock Analytic cycles on complex manifolds.
\newblock {\em Topology}, 1(1):25--45, 1962.

\bibitem{bake}
M.~Baker.
\newblock Specialization of linear systems from curves to graphs.
\newblock {\em Algebra \& Number Theory}, 2(6):613--653, 2008.

\bibitem{bala2}
V.~Balaji, A.~D. King, and P.~E. Newstead.
\newblock Algebraic cohomology of the moduli space of rank 2 vector bundles on
  a curve.
\newblock {\em Topology}, 36(2):567--577, 1997.

\bibitem{indpre}
S.~Basu, A.~Dan, and I.~Kaur.
\newblock Degeneration of intermediate {J}acobians and the {T}orelli theorem.
\newblock {\em Documenta Mathematica}, 24:1739--1767, 2019.

\bibitem{genpre}
S.~Basu, A.~Dan, and I.~Kaur.
\newblock Generators of the cohomology ring, after {N}ewstead.
\newblock {\em Proceedings of the American Mathematical Society},
  150(06):2569--2577, 2022.

\bibitem{birk}
C.~Birkenhake and H.~Lange.
\newblock {\em Complex abelian varieties}, volume 302.
\newblock Springer Science \& Business Media, 2013.

\bibitem{DK2}
A~Dan and I.~Kaur.
\newblock N{\'e}ron models of intermediate {J}acobians associated to moduli
  spaces.
\newblock {\em Revista Matem{\'a}tica Complutense},
  DOI-10.1007/s13163-019-00333-y:1--26, 2019.

\bibitem{mumf}
A.~Dan and I.~Kaur.
\newblock Generalization of a conjecture of {M}umford.
\newblock {\em Advances in Mathematics}, 383:107676, 2021.

\bibitem{deba}
O.~Debarre.
\newblock {\em Higher-dimensional algebraic geometry}.
\newblock Springer Science \& Business Media, 2013.

\bibitem{delibla}
P.~Deligne.
\newblock Th{\'e}oreme de {L}efschetz et criteres de d{\'e}g{\'e}n{\'e}rescence
  de suites spectrales.
\newblock {\em Publications Math{\'e}matiques de l'Institut des Hautes
  {\'E}tudes Scientifiques}, 35(1):107--126, 1968.

\bibitem{deli2}
P.~Deligne.
\newblock {\em {\'E}quations diff{\'e}rentielles {\`a} points singuliers
  r{\'e}guliers}, volume 163.
\newblock Springer, 2006.

\bibitem{huy}
D.~Huybrechts and M.~Lehn.
\newblock {\em The geometry of moduli spaces of sheaves}.
\newblock Springer, 2010.

\bibitem{ind}
I.~Kaur.
\newblock The ${C_1}$ conjecture for the moduli space of stable vector bundles
  with fixed determinant on a smooth projective curve,.
\newblock \url{https://refubium.fu-berlin.de/handle/fub188/6611}, 2016.
\newblock Ph.D. thesis, Freie University Berlin.

\bibitem{kingn}
A.~D. King and P.~E. Newstead.
\newblock On the cohomology ring of the moduli space of rank 2 vector bundles
  on a curve.
\newblock {\em Topology}, 37(2):407--418, 1998.

\bibitem{kuli}
V.~S. Kulikov.
\newblock {\em Mixed Hodge structures and singularities}, volume 132.
\newblock Cambridge University Press, 1998.

\bibitem{lant1}
A.~Langer.
\newblock Moduli spaces of principal bundles on singular varieties.
\newblock {\em Kyoto Journal of Mathematics}, 53(1):3--23, 2013.

\bibitem{lewis}
J.~D. Lewis and B.~B. Gordon.
\newblock {\em A survey of the {H}odge conjecture}, volume~10.
\newblock American Mathematical Soc., 2016.

\bibitem{mumn}
D.~Mumford and P.~Newstead.
\newblock Periods of a moduli space of bundles on curves.
\newblock {\em American Journal of Mathematics}, 90(4):1200--1208, 1968.

\bibitem{MOS}
V.~Mu{\~n}oz, A.~G. Oliveira, and J.~S{\'a}nchez~Hern{\'a}ndez.
\newblock Motives and the {H}odge conjecture for moduli spaces of pairs.
\newblock {\em Asian journal of mathematics}, 19(2):281--306, 2015.

\bibitem{nagsesh}
D.~S. Nagaraj and C.~S. Seshadri.
\newblock Degenerations of the moduli spaces of vector bundles on curves {II}
  ({G}eneralized {G}ieseker moduli spaces).
\newblock In {\em Proceedings of the Indian Academy of Sciences-Mathematical
  Sciences}, volume 109, pages 165--201. Springer, 1999.

\bibitem{narafac}
M.~S. Narasimhan and T.~R. Ramadas.
\newblock Factorisation of generalised theta functions. i.
\newblock {\em Inventiones mathematicae}, 114(1):565--623, 1993.

\bibitem{new1}
P.~E. Newstead.
\newblock Characteristic classes of stable bundles of rank 2 over an algebraic
  curve.
\newblock {\em Transactions of the American Mathematical Society},
  169:337--345, 1972.

\bibitem{pan}
R.~Pandharipande.
\newblock A compactification over {$M_g$} of the {U}niversal {M}oduli {S}pace
  of {S}lope-{S}emistable vector bundles.
\newblock {\em Journal of the American Mathematical Society}, 9(2):425--471,
  1996.

\bibitem{pet}
C.~Peters and J.~H.~M. Steenbrink.
\newblock {\em Mixed {H}odge structures}, volume~52.
\newblock Springer Science \& Business Media, 2008.

\bibitem{pezz}
G.~Pezzini.
\newblock Lectures on spherical and wonderful varieties.
\newblock {\em Les cours du CIRM}, 1(1):33--53, 2010.

\bibitem{schvar}
W.~Schmid.
\newblock Variation of {H}odge structure: the singularities of the period
  mapping.
\newblock {\em Inventiones mathematicae}, 22(3-4):211--319, 1973.

\bibitem{ste1}
J.~Steenbrink.
\newblock Limits of {H}odge structures.
\newblock {\em Inventiones mathematicae}, 31:229--257, 1976.

\bibitem{ste2}
J.~H.~M. Steenbrink.
\newblock {\em Mixed Hodge structure on the vanishing cohomology}.
\newblock Department of Mathematics, University of Amsterdam, 1976.

\bibitem{sun1}
X.~Sun.
\newblock Moduli spaces of {SL}(r)-bundles on singular irreducible curves.
\newblock {\em Asia Journal of Mathematics}, 7(4):609--625, 2003.

\bibitem{tha}
M.~Thaddeus.
\newblock {\em Algebraic geometry and the Verlinde formula}.
\newblock PhD thesis, University of Oxford, 1992.

\bibitem{v4}
C.~Voisin.
\newblock {\em {H}odge Theory and Complex Algebraic Geometry-I}.
\newblock Cambridge studies in advanced mathematics-76. Cambridge University
  press, 2002.

\bibitem{voiconj}
C.~Voisin.
\newblock Some aspects of the {H}odge conjecture.
\newblock {\em Japanese Journal of Mathematics}, 2(2):261--296, 2007.

\end{thebibliography}
\end{document}